\documentclass[11pt]{article}

\usepackage{macros}

\title{$\L8$ spaces and derived loop spaces}

\author{Ryan Grady and Owen Gwilliam}

\begin{document}

\maketitle

\begin{abstract}
We develop further the approach to derived differential geometry introduced in Costello's work on the Witten genus \cite{WG2}. In particular, we introduce several new examples of $\L8$ spaces, discuss vector bundles and shifted symplectic structures on $\L8$ spaces, and examine in some detail the example of derived loop spaces. This paper is background for \cite{GGqft}, in which we define a quantum field theory on a derived stack, building upon Costello's definition of an effective field theory in \cite{Cos1}.
\end{abstract}

\tableofcontents

\section{Introduction and overview}

In his work constructing the Witten genus with a two-dimensional sigma model \cite{WG2}, Costello introduces a framework for derived differential geometry. His approach is functorial in nature: he defines a {\em derived stack} as a functor from a category of test objects to the category of simplicial sets, satisfying some conditions familiar from geometry. His test objects are a modest enlargement of smooth manifolds to include nilpotent and derived-nilpotent directions, which is why this approach is a version of derived differential geometry. The conditions characterizing his derived stacks are mild.\footnote{It would be interesting to work out the analogs of Deligne-Mumford and Artin stacks in this setting, as Costello's conditions simply pick out homotopy sheaves and hence encompass functors that are not very geometric. It would also be quite useful and interesting to compare this approach to the approaches of Carchedi-Roytenberg \cite{CarchediRoytenberg}, Borisov-Noel \cite{BorisovNoel}, Lurie \cite{LurieDAGV}, Spivak \cite{Spivak}, Schreiber \cite{Schreiber}, and Joyce \cite{Joyce}.} Moreover, he points out a nice subset of derived stacks that have a kind of ``ringed space" description, which he calls $\L8$ spaces, that exploit the intimate relationship between dg Lie algebras and deformation theory. The appeal of $\L8$ spaces is that they allow one to do explicit computations very efficiently, particularly computations that appear in (perturbative) quantum field theory (QFT).

Our goal here is to verify some simple properties of Costello's formalism and to point out some appealing features. Much of the paper simply elaborates on Costello's work in \cite{WG2}. In a companion paper \cite{GGqft}, we articulate how these definitions allow one to extend the reach of his  formalism for perturbative quantum field theory to encompass less-perturbative QFT, developing explicitly some ideas that are implicit in Costello's work. 

\subsection{An overview of the paper}

We begin by developing the definition of a derived stack, in Costello's sense. We then introduce a special class of derived stacks, the $\L8$ spaces, which can be thought of as families of formal moduli problems parametrized by smooth manifolds. Thereafter, we introduce some simple examples of $\L8$ spaces and a modicum of geometry on $\L8$ spaces (notably vector bundles and shifted symplectic forms). Finally, we discuss various kinds of derived loop spaces that are relevant to our work on a 1-dimensional sigma model \cite{GG}.

\subsection{The motivation from physics}

Before delving into the text itself, we remark on the motivations behind this approach. The immediate impetus arose from Costello's goal of encoding a nonlinear sigma model into his formalism for quantum field theory \cite{Cos1}.  The reader interested just in derived geometry is welcome to skip this short discussion. As we have not yet defined any of the principal objects of our formalism, we will speak loosely in geometric language.

Let $M$ be the smooth manifold on which the fields of our field theory will live. In his book, Costello defined families of field theories over {\em nilpotent dg manifolds}. More explicitly, for $\cN$  a nilpotent dg manifold, suppose we have an $M$-fiber bundle $M \hookrightarrow P \to \cN$ and a relative vector bundle $V \to P$.  Costello's formalism of renormalization works in this relative situation, where for each point $x$ in $\cN$, we have a field theory on $V_x \to M_x$ whose underlying fields are the smooth sections of the fiberwise vector bundle. 

Thus Costello needed to rephrase a nonlinear sigma model --- a field theory whose fields are $\Maps(M,X)$, the space of smooth maps from $M$ into a manifold $X$ --- within such a framework, where fields are sections of a vector bundles. There is a standard idea from physics that suggests how to do this, and $\L8$ spaces allow one to formulate it as a  mathematical procedure. 

We would like to compute the path integral 
\[
\int_{\phi \in \Maps(M,X)} e^{-S(\phi)/\hbar} \,\sD \phi,
\] 
or -- more accurately -- provide an asymptotic expansion for this integral in the regime where $\hbar$ is infinitesimally small. For very small $\hbar$, the measure should be concentrated in a small tubular neighborhood $Tub$ of $Sol$, the subspace cut out by the solutions to the equations of motion. Thus, the integral can be well-approximated by pulling back the measure to the neighborhood $Tub$. We then identify this tubular neighborhood with the normal bundle $N$ to $Sol$ and compute the approximate integral in two steps. First, we use perturbative methods to compute the fiberwise integral and obtain a measure on $Sol$. Second, we integrate over $Sol$ itself. This integral over $Sol$ breaks up into a sum over the connected components of $Sol$.

Among the connected components of $Sol$, the lowest-energy solutions for a sigma model are given typically by the constant maps, and this component of the space of solutions looks like a copy of the target manifold $X$ itself. This component usually provides the dominant term in the sum over components. Hence, our path integral should be well-approximated by integrating over a tubular neighborhood $T$ just of the constant maps $X \hookrightarrow \Maps(M,X)$. Note that a really small perturbation of such a constant map $\underline{x}: M \to x \in X$ can be viewed as a map from $M$ to a small ball in $X$ around $x$. 

In physics, one often applies this heuristic idea as follows. One views $T$ as a vector bundle over $X$ whose fiber at $x \in X$ is $\Gamma(M, \underline{x}^* T_x X)/T_x X$. (This vector space takes the quotient of all the smooth sections of the pullback tangent bundle --- which are the ways of wiggling the constant map $x$ --- by the subspace of sections that just move to a nearby constant map.) Formal derived geometry provides a nice mathematical language for describing the formal neighborhood $\cT$ of the constant maps inside $\Maps(M,X)$. Indeed, Costello showed that $\cT$ is an $\L8$ space. He then showed that his perturbative formalism interacted cleanly with $\L8$ spaces to realize the two-step process of integration. In the language of \cite{WG2}, the quantum BV theory he produces out of $\cT$ provides a \emph{projective volume form}. 

In our companion paper, we explain in more depth both the heuristic picture of path integral quantization of the nonlinear sigma model and the precise realization of that idea using Costello's pair of formalisms (for QFT and derived geometry).

\subsection{A comment on our imagined audience}

Our primary purpose for this formalism is the construction of quantum field theories, and we imagine that some of our audience has a working knowledge of homological algebra and category theory at the level of Weibel \cite{Weibel}  and some familiarity with the language of $\L8$ algebras, due to their role in deformation quantization {\it \`a la} Kontsevich \cite{KonDQ}. We have thus sought to explain constructions or concepts, like homotopy limits, that we use that fall outside those references. Those readers comfortable with homotopical algebra are encouraged to skip over such discussions.

\subsection{Acknowledgements}

First and foremost, we thank Kevin Costello for many helpful conversations and for introducing us to these ideas in the context of field theory. Our understanding on derived geometry, however limited, is rooted in discussions with John Francis, David Nadler, Anatoly Preygel, and Nick Rozenblyum. The influence of the ideas of Dennis Gaitsgory, Jacob Lurie, Bertrand To\"en, and Gabriele Vezzosi should also be clear: we thank them for their inspiring ideas, lectures, and texts. We have also benefited from conversations with many people on this topic, including David Ayala, David Carchedi, Lee Cohn, Vasiliy Dolgushev, Si Li, Thel Seraphim, Yuan Shen, Jim Stasheff, Mathieu Stienon, Stephan Stolz, Peter Teichner, Ping Xu, and Brian Williams.


\subsection{Notation}

\begin{itemize}
\item For $M$ a smooth manifold, let $\Omega^*_M$ denote the sheaf of differential forms on $M$ as a sheaf of commutative dg algebras. Let $\CC_M$ denote the locally constant sheaf assigning $\CC$ to any connected open.
\item For $\cC$ a category, we denote the set of morphisms from $x$ to $y$ by $\cC(x,y)$.
\item For $A$ a cochain complex, $A^\sharp$ denotes the underlying graded vector space. If $A$ is a cochain complex whose degree $k$ space is $A^k$, then $A[1]$ is the cochain complex where $A[1]^k = A^{k+1}$. 
\item For $A$ a dg module over a dg algebra $R$, the \emph{dual} of $A$, denoted $A^\vee$, means the graded dual, whose $k$th component $(A^\vee)^k$ is the set of degree $-k$ elements of $\Hom_{R^\sharp}(A^\sharp, R^\sharp)$. The differential is determined by requiring that the evaluation pairing ${\rm ev}: A^\vee \otimes_{R^\sharp} A \to R$ be a map of $R$-modules. 
\item Let $\Sets$ denote the category whose objects are sets and whose morphisms are functions between sets.
\item Let $\Delta$ denote the (finite) {\em ordinal category}, in which an object is a totally ordered finite set and a morphism is a nondecreasing function. We will usually restrict attention to the skeletal subcategory with objects
\[
[n] := \{0 < 1< \cdots < n\}
\]
and morphisms $f: [m] \to [n]$ such that $f(i) \leq f(j)$ for $i \leq j$.
\item Let $\ssets$ denote the category of simplicial sets, namely the category of functors $\Fun(\Delta^{op},\Sets)$. A simplicial set $X$ will often be written as $X_\bullet$, and $X_n := X([n])$ denotes the ``set of $n$-simplices of $X$."
\item We denote by $\triangle[m]$ the simplicial set $\Delta(-,[m])$. Under the Yoneda lemma, this is the functor represented by the object $[m]$.\footnote{Note that we use the Greek letter $\Delta$ to denote the category and the triangle symbol $\triangle$ to denote a simplex.}
\item Let $\csets$ denote the category {\em co}simplicial sets $\Fun(\Delta,\Sets)$. A cosimplicial set $Y$ will often be written as $Y^\bullet$, and $Y^n := Y([n])$ denotes the ``set of $n$-cosimplices of $Y$."
\item Let $cs\!\Sets$ denote the category of cosimplicial simplicial sets $\Fun(\Delta \times \Delta^{op},\Sets)$. We will often denote a cosimplicial simplicial set $Z$ by $Z^\bullet_\bullet$, and $Z^n_\bullet$ is the simplicial set of $n$-cosimplices and $Z^\bullet_n$ is the cosimplicial set of $n$-simplices.
\item By $\triangle^n$ we mean the standard $n$-simplex in $\RR^n$.
\item To indicate the end of an example or remark, we use the symbol $\Diamond$, just as we use $\Box$ to indicate the end of a proof.
\end{itemize}

\section{The category of test objects}

Let $\Man$ denote the category of smooth, finite-dimensional manifolds (without boundary) and smooth morphisms. From here on, {\em manifold} will mean smooth and finite-dimensional. 

We ``enlarge'' the category $\Man$ by allowing a certain class of structure sheaves following ideas and constructions of Costello \cite{Cos1}, \cite{WG2}. 
 
\begin{definition}
On a smooth manifold $M$, let 
\[
\Sym(T^*_M[-1]) = \bigoplus_{n = 0}^{\dim M} \bigwedge^n T^*_M [-n]
\] 
denote the $\ZZ$-graded vector bundle whose smooth sections are the differential forms on $M$ (with their usual cohomological degree). Note that $\Sym(T^*_M[-1])$ is a graded-commutative algebra in the category of vector bundles on $M$ by (fiberwise) wedge product. This product induces the wedge product on its smooth sections, namely the wedge product of differential forms.
\end{definition}

\begin{definition}
A {\em nilpotent dg manifold} $\cM$ is a triple $(M, \sO_\cM, \sI_\cM)$ consisting of the following data and conditions.
\begin{enumerate}
\item[(1)] A manifold $M$. 
\item[(2)] A $\ZZ$-graded vector bundle $A \to M$ of total finite rank equipped with 
\begin{enumerate}
\item[a)] the structure of a module over $\Sym(T^*_M[-1])$;
\item[b)] a fiberwise multiplication map $m: A \ot A \to A$ making $A$ into a unital graded-commutative algebra in vector bundles on $M$,\footnote{That is, $m$ is a map of vector bundles from the Whitney tensor product $A \ot A$ to $A$.} so that $(A,m)$ is an algebra over $\Sym(T^*_M[-1])$;
\item[c)] there is a short exact sequence of vector bundles $I \hookrightarrow A \to A/I \cong \CC \times M$ (i.e., the quotient is the trivial rank 1 bundle) that respects the multiplication $m$ and an associated chain of vector bundles
\[
0 =I^{n+1} \subset I^n \subset I^{n-1} \subset \cdots \subset I
\]
compatible with multiplication (i.e., $I^k \cdot I^\ell \subset I^{k+\ell}$). That is, $I$ forms a nilpotent ideal in $A$.
\end{enumerate}
Let $\sO_\cM$ denote the sheaf of smooth sections of the algebra bundle $A$. We equip $\sO_\cM$ with a derivation of cohomological degree 1 so that it is a sheaf of unital commutative dg $\Omega^*_M$-algebras.
\item[(3)] A map $q: \sO_\cM \to \cinf_M$ of sheaves of $\Omega^*_M$-algebras whose kernel is the sheaf $\sI_\cM$ of smooth sections of $I$. This sheaf $\sI_\cM$ forms a nilpotent dg ideal.\footnote{We will generally suppress the map $q$ from notation as it is nearly always obvious from context.}
\item[(4)] We require the cohomology of $\sO_\cM(U)$ to be concentrated in non-positive degrees for sufficiently small open sets $U \subset M$.
\end{enumerate}
\end{definition}

Note that the conditions on the sheaf $\sO_\cM$ imply that its global cohomology $\sO_\cM(M)$ lives in finitely many cohomological degrees.  We call $\sO_\cM$ the {\em structure sheaf} of the nilpotent dg manifold $\cM$.

\begin{definition}
Let $\cN$ be a nilpotent dg manifold with underlying manifold $N$ and graded vector bundle $B$, and let $\cM$ be a nilpotent dg manifold with underlying manifold $M$ and graded vector bundle $A$. A {\em map of nilpotent dg manifolds} $F: \cN \to \cM$ is a pair $(f, \phi)$ where $f$ is a smooth map from $N$ to $M$ and $\phi$ is a map of graded vector bundles from the pullback bundle $f^{-1} A$ to $B$ such that there is a commuting diagram of commutative dg $f^{-1}\Omega_M$-algebra sheaves (on $N$) 
\[
\xymatrix{
f^{-1}\sO_\cM \ar[r]^\phi \ar[d] & \sO_\cN \ar[d] \\
f^{-1} \cinf_M \ar[r] & \cinf_N
}.
\]
(In other words, the map of bundles is compatible with the graded algebra structure on the bundles, the differentials on the sheaf of sections, and the filtrations by the nilpotent dg ideal sheaves.)
\end{definition}

There is one last construction that we will need.

\begin{definition} 
Let $\cM = (M, \sO_\cM)$ and $\cN = (N, \sO_\cN)$ be nilpotent dg manifolds. The \emph{product} $\cM \times \cN$ is the nilpotent dg manifold $(M \times N, \sO_{\cM \times \cN}, \sI_{\cM \times \cN})$, where 
\[
\sO_{\cM \times \cN} := \lim_{\stackrel{\longleftarrow}{i,j}} \left(\pi_M^{-1} \sO_\cM/\sI_\cM^i \right) \otimes_{\Omega^*_{M\times N}} \left(\pi_N^{-1} \sO_\cN/\sI_\cN^j\right)
\]
equipped with the nilpotent ideal
\[
\sI_{\cM \times \cN} = \left(\pi_M^{-1} \sI_\cM  \otimes_{\Omega^*_{M\times N}} \pi_N^{-1} \sO_\cN \right) \oplus \left(\pi_M^{-1} \sO_\cM\otimes_{\Omega^*_{M\times N}} \pi_N^{-1} \sI_\cN\right).
\]
\end{definition}

This definition arises by requiring the natural compatibility with the filtrations.

\begin{remark}
Our structure sheaves are always given as sections of graded vector bundles, equipped with extra structure, although it might seem more natural to work immediately with sheaves. This vector bundle condition appears in large part to ensure compatibility with the quantum field theory constructions in \cite{Cos1}.  In our future work, \cite{GGqft}, we use the present results to construct effective field theories and quantization in families, where the parameterizing spaces are exactly the nilpotent dg manifolds we just defined.
\end{remark}

\subsection{Relation to smooth manifolds, formal geometry, complex geometry, and foliations}

To give a sense of this new category $\dgMan$ of nilpotent dg manifolds, we now describe some important examples.

\begin{example}
For $M$ a smooth manifold, the nilpotent dg manifold $(M,\Omega^*_M)$ is known as the {\em de Rham space} of $M$ and is denoted by $M_{dR}$. By definition, ever other nilpotent dg manifold $\cM$ with $M$ as the underlying smooth manifold possesses a distinguished map $\cM \to M_{dR}$. 
\hfill $\Diamond$ 
\end{example}

\begin{remark}
There are several ways to think about the role of $M_{dR}$. Observe that the simplest natural sheaf of algebras to put on $M$ is $\CC_M$, the locally constant sheaf, but this sheaf is not soft and not well-suited to the techniques of differential geometry. The de Rham complex is then a pleasant replacement for $\CC_M$, since it locally recovers $\CC_M$ as its cohomology. Note that the structure sheaves of nilpotent dg manifolds are algebras over the de Rham complex, and thus well-behaved replacements for algebras over the constant sheaf. 

Another aspect is more categorical in nature. If we were to define a category of $\sO$-modules over $M_{dR}$ (although we will not develop such a formalism here), it should be equivalent to the category of $D$-modules on $M$. Heuristically, this relationship is clearest when one thinks about vector bundles on $M_{dR}$: these are essentially vector bundles with flat connection on the smooth manifold $M$. The role of $M_{dR}$ is to provide the natural space over which live all ``things with a flat connection over $M$" (or more accurately, things with a system of differential equations over $M$). Indeed, we view the work of Block and collaborators \cite{Block}, \cite{BBandBlock}, \cite{BlockSmith} as an approach to this question. (Works of Kapranov \cite{Kapranov} and of Simpson-Teleman \cite{SimpsonTeleman} are also highly relevant here.)

Finally, in lemma \ref{lem:dR}, we indicate how this notion of de Rham space connects with the de Rham space in algebraic geometry. 
\hfill $\Diamond$ 
\end{remark}

\begin{example}
Let $M$ be a smooth manifold. Then $(M,\cinf_M)$ is a nilpotent dg manifold where $\cinf_M$ is viewed as an $\Omega^\ast_M$ module by the quotient map $\Omega^*_M \to \Omega^0_M$ whose kernel is the differential ideal generated by the 1-forms. We denote it by $M_{sm}$.
\hfill $\Diamond$ \end{example}

It is straightforward to see the following, which explains why $\dgMan$ is an enlargement of the natural test objects for smooth geometry.

\begin{prop}
The inclusion functor $i: \Man \to \dgMan$ where $M \mapsto M_{sm} = (M,\cinf_M)$ is fully faithful.
\end{prop}

In $\dgMan$, the de Rham space of a smooth manifold $M$ represents the set of constant maps.  More precisely we have the following.

\begin{prop} Let $N$ and $M$ be  smooth manifolds.  The set of maps $\dgMan (N_{dR} , M_{sm})$ is in bijection with the underlying set of $M$ (viewed as the constant maps from $N$ to $M$).
\end{prop}

\begin{proof}
A map of nilpotent dg manifolds $N_{dR} \to (M, \cinf_M)$ consists of a smooth map $f : N \to M$ and a map of sheaves of dg $f^{-1} \Omega_M$ algebras $\phi : f^{-1} \cinf_M \to \Omega^\ast_N$ such that

\[
\xymatrix{
f^{-1}\cinf_M \ar[r]^\phi \ar[d]_{\mathrm{Id}} & \Omega^\ast_N \ar[d] \\
f^{-1} \cinf_M \ar[r] & \cinf_N
}
\]
commutes. In particular, the map $\phi$ commutes with differentials.  The differential on the structure sheaf $(M, \cinf_M)$ is trivial, and hence we see that the pullback of a smooth function on $M$ must be a constant function.
\end{proof}

\begin{remark}
By definition, every nilpotent dg manifold $\cM = (M, \sO_\cM)$ ``lives between" $M_{sm}$ and $M_{dR}$ in the sense that we have canonical maps
\[
M_{sm} \to \cM \to M_{dR},
\]
where the underlying map of manifolds is the identity and the algebra maps are provided by definition.
\hfill $\Diamond$ \end{remark}

\begin{example}\label{ex:artinian}
Let $R$ be an artinian algebra over $\CC$, such as the {\em dual numbers} $\CC[\epsilon]/(\epsilon^2)$. Then $(\pt, R)$ is a nilpotent dg manifold.
\hfill $\Diamond$ \end{example}

Let $\Art_\CC$ denote the category of artinian algebras over $\CC$. The opposite category $\Art_\CC^{op}$ encodes ``fat points." It's easy to see that as the underlying manifold is just a point we have a full and faithful embedding. 

\begin{prop}
The $\Spec$ functor 
\[
\begin{array}{ccc}
\Art^{op}_\CC & \to & \dgMan \\
A & \mapsto & \Spec A = (\pt, A)
\end{array}
\]
is fully faithful.
\end{prop}

We can also include with dg artinian algebras concentrated in nonpositive degrees, such as the {\em shifted dual numbers} $\CC[\epsilon]/(\epsilon^2)$ where $\epsilon$ has negative cohomological degree. Thus, $\dgMan$ includes the basic ingredient of {\em derived} deformation theory. 

This proposition also indicates one way that formal deformation theory will relate to $\dgMan$, since we can fatten any manifold this way.  For any smooth manifold $M$, we can ``thicken" $M$ by a ``$\Spec R$-bundle'' to obtain an interesting nilpotent dg manifold. (That is, let the structure sheaf be sections of an $R$-bundle over $M$.) 

\begin{example}
Given a complex structure on a smooth manifold $M$, there is a nilpotent dg manifold encoding the complex manifold, namely $(M,\Omega^{0,\ast}_M)$, where $\Omega^{0,\ast}_M$ is the Dolbeault complex for this complex structure. We view $\Omega^{0,\ast}_M$ as the quotient of $\Omega^\ast_M$  by the differential ideal generated by the $(1,\ast)$-forms.
\hfill $\Diamond$ \end{example}

\begin{prop}
The inclusion functor is a fully faithful embedding from the category of complex manifolds into $\dgMan$.
\end{prop}

\begin{example}
Let $\cF$ be a (regular) foliation of $M$. Equivalently, let $\rho: T_{\cF} \hookrightarrow T_M$ be a subbundle of the tangent bundle that is \emph{integrable}: the Lie bracket of any two sections of $T_{\cF}$ is always a section of $T_{\cF}$. Thus, $\cF$ provides a Lie algebroid $T_{\cF}$, and a standard construction for the theory of Lie algebroids then provides a nilpotent dg manifold, as follows. (See, for example, \cite{Rinehart}, \cite{Mackenzie}, \cite{MehtaCoh}, \cite{AbadCrainic}.)

The \emph{Chevalley-Eilenberg cochain complex} of $T_{\cF}$ is a sheaf of commutative dg algebras determined by the Lie algebroid. We denote it $C^* T_{\cF}$. The underlying sheaf of graded algebras $C^\sharp T_{\cF}$ is given by the smooth sections of the bundle $\Sym(T^*_{\cF}[-1]))$, and multiplication is the pointwise wedge product. Hence it is a graded algebra over $C^\infty$, but we equip it with a differential that is not $C^\infty$-linear.\footnote{The differential is determined by the Lie bracket, which is not $C^\infty$-linear.} The differential $d$ is determined by the following conditions. First, for any function $f$, viewed as an element of $C^0 T_{\cF}(M)$, we have that
\[
d(f)(X) = \rho(X)(f) 
\]
for every $X \in \Gamma(M,T_{\cF})$. Second for any $\alpha \in C^1 T_{\cF}(M)$,
\[
d(\alpha)(X \wedge Y) = \rho(X)(\alpha(Y)) - \rho(Y)(\alpha(X)) - \alpha([X,Y]),
\]
for all $X, Y \in \Gamma(M,T_{\cF})$. Third, we require that $d^2 = 0$ and 
\[
d(\alpha \wedge \beta) = (d\alpha) \wedge \beta + (-1)^{\alpha} \alpha \wedge (d\beta)
\]
for all elements $\alpha$ and $\beta$.\footnote{This construction is a systematic generalization of the de Rham complex: when $T_{\cF} = T_M$, the Chevalley-Eilenberg complex is precisely $\Omega^\ast(M)$.} Observe that for $U$ an open on which the foliation decomposes as $U \cong \RR^p \times \RR^{\dim M - p}$, where the leaves are codimension $p$, then
\[
H^k(C^* T_{\cF}(U)) = \begin{cases} C^\infty(\RR^p), & k = 0 \\ 0, \text{else} \end{cases}.
\]
(This is a direct consequence of the usual Poincar\'e lemma.) Note that $C^* T_{\cF}$ has a nilpotent dg ideal given by $C^{\geq 1} T_{\cF}$. Lastly, note that $C^* T_{\cF}$ is a dg $\Omega^*_M$-algebra -- indeed a quotient algebra -- via the algebra map determined by the dual to the anchor map $\rho^\ast : T^*_M \to T^*_{\cF}$.

Let $M_{\cF}$ denote the nilpotent dg manifold $(M, C^* T_{\cF})$. It provides a dg manifold describing the ``derived leaf space'' of the foliation, as the Lie algebroid cohomology is precisely the derived functor for taking invariants of functions along leaves.

This construction encompasses several earlier examples: when $T_\cF = 0$, we recover $M_{sm}$; when $T_{\cF} = T_M$, we recover $M_{dR}$; and when $M$ is a complex manifold, $M_{\delbar}$ is associated to the foliation given by $T^{0,1}_M$.
\hfill $\Diamond$ 
\end{example}

\begin{prop}
The functor from the category of regular foliations to $\dgMan$ sending $\cF$ to $M_\cF$ is a fully faithful embedding.
\end{prop}

\subsection{A notion of weak equivalence}

Costello introduces an interesting notion of weak equivalence between nilpotent dg manifolds. His notion relies on the existence of a natural filtration on the structure sheaf $\sO_\cM$ of a nilpotent dg manifold $\cM$. In particular, let $\sI_\cM = \ker q$ denote the sheaf of nilpotent dg ideals in $\sO_\cM$. We have the filtration
\[
F^k \sO_\cM = \sI^k_\cM.
\] 
Let $\Gr \sO_\cM$ denote the associated graded dg algebra.

\begin{definition}
A map $F: \cM \to \cN$ in $\dgMan$ is a {\em weak equivalence} if
\begin{enumerate}
\item[(1)] the smooth map $f: M \to N$ is a diffeomorphism, and
\item[(2)] the map of commutative dg algebras $\Gr \phi: f^{-1}\Gr \sO_\cN \to \Gr \sO_\cM$ is a quasi-isomorphism.
\end{enumerate}
\end{definition}

This definition provides a well-behaved notion of weak equivalence because 
\begin{itemize}
\item every isomorphism of nilpotent dg manifolds (i.e., a diffeomorphism with a strict isomorphism of structure sheaves) is a weak equivalence, and 
\item the notion satisfies the 2-out-of-3 property because diffeomorphism and quasi-isomorphism do. 
\end{itemize}
Nonetheless, this definition might look a little strange because of the role played by the associated graded algebra. Observe, however, that it implies that $\phi$ is a quasi-isomorphism: $\phi$ preserves the natural filtration on the structure sheaves and hence induces a map of spectral sequences that is a quasi-isomorphism on the first page. Thus the definition is {\em stronger} than requiring that $\phi$ is a quasi-isomorphism, which might be the first definition that comes to mind. This stronger condition depends crucially on the filtration.

As further motivation for the definition, we note that the role of nilpotent dg manifolds here is supposed to parallel the role of artinian algebras in formal deformation theory.  Arguments in deformation theory often proceed by {\em artinian induction}: every local artinian $\CC$-algebra $(A,\frak{m})$ possesses a canonical tower of quotients
\[
A \to A/\frak{m}^n \to \cdots \to A/\frak{m}^2 \to A/\frak{m} \cong \CC,
\]
and at each step we extend an artinian algebra by a square-zero ideal so it suffices to prove some property holds for such small extensions
\[
I \hookrightarrow B \twoheadrightarrow A,
\]
where $I$ is the kernel of a ring map and $I^2 = 0$ inside $B$. The filtration on $\sO_\cM$ is our substitute for this canonical tower of quotients, and we will prove our main theorem by using a version of artinian induction.

\section{Derived stacks}

Our notion of ``derived stack" or ``derived space" will be, as usual, a kind of sheaf of simplicial sets on the site of ``test objects'' (cf., \cite{TV}). Thus we need to equip $\dgMan$ with a site structure.

Recall that the category $\Man$ has a site structure where a covering is simply an open cover in the usual sense. A {\em covering} of $\cM = (M,\sO_\cM)$ in $\dgMan$ is a collection $\{ \cU_i = (U_i, \sO_i)\}$ of nilpotent dg manifolds with maps $\{ F_i:  \cU_i \to \cM\}$ such that the collection $\{U_i\}$ forms a cover of $M$ in $\Man$ and the maps of structure sheaves $\phi_i: f^{-1}_i\sO_\cM \to \sO_i$ are isomorphisms.

\begin{definition}
A {\em derived stack} is a functor $\bX: \dgMan^{op} \to s\!\Sets$ satisfying
\begin{enumerate}
\item[(1)] $\bX$ sends weak equivalences of nilpotent dg manifolds to weak equivalences of simplicial sets, and
\item[(2)] $\bX$ satisfies \v{C}ech descent (see below for a reminder on what this means).
\end{enumerate} 
\end{definition}

Note that a derived stack is merely a homotopical kind of sheaf, and thus the definition does not capture any particularly {\em geometric} properties. For instance, we do not require $\bX$ to locally resemble a ringed space or orbifold. Our focus in this paper is on a class of examples, the $\L8$ spaces, that {\em do} have a very geometric flavor. (It would be interesting to work out the analogs of orbifold or Artin stack in this context.)

\begin{definition}
Let $\bX, \bY : \dgMan^{op} \to \ssets$ be derived stacks. A \emph{map of stacks} $\alpha: \bX \to \bY$ is just a natural transformation between the functors. A \emph{weak equivalence of stacks} is a map of stacks $\alpha$ such that $\alpha(\cM): \bX(\cM) \to \bY(\cM)$ is a weak equivalence for every nilpotent dg manifold $\cM$.
\end{definition}

Let $dSt$ denote the category of derived stacks. Hence, we will denote the morphisms from a derived stack $\bX$ to $\bY$ by $dSt(\bX,\bY)$. 

\subsection{\v{C}ech descent and homotopy sheaves}

We recall the usual notion of a sheaf before giving the souped-up version we need.

Let $X$ be a topological space and let ${\rm Opens}_X$ denote the poset category whose objects are opens in $X$ and whose morphisms are inclusions of opens. A {\em presheaf} on $X$ with values in the category $\sC$ is a functor $\cF: {\rm Opens}^{op}_X \to \sC$. A {\em sheaf} is a presheaf such that for any open $U$ and any cover $\fV = \{V_i\}$ of $U$, we have
\[
\cF(U) \cong \mathrm{eq} \left( \prod_i \cF(V_i) \rightrightarrows \prod_{i,j} \cF(V_i \cap V_j) \right),
\]
where $\mathrm{eq}$ denotes ``equalizer" and the two arrows are the two natural restriction maps for $\cF$.

In our setting, the value category $\sC$ is the category $s\!\Sets$, and we want to view two simplicial sets as the same if they are weakly equivalent. Thus, wherever we would ordinarily compute (co)limits, we should work with {\em homotopy} (co)limits instead. Moreover, we don't merely want to require agreement on overlaps; we want to coherently agree on overlaps-of-overlaps and so on.

These desires indicate how we should refine the notion of sheaf. 

\begin{definition}
Let $\cM = (M, \sO_\cM)$ be a nilpotent dg manifold. Let $\fV = \{V_i,  \; : \; i \in I\}$ be a cover for the underlying smooth manifold $M$. We also use $\fV$ to denote, abusively, the nilpotent dg manifold $\coprod_{i \in I} (V_i, \sO_\cM |_{V_i})$, given by disjoint union over the opens in the cover.  Let $\check{C}\fV_\bullet$ denote the simplicial nilpotent dg manifold whose $n$-simplices are
\[
\check{C}\fV_n := \fV \times_M \cdots \times_M \fV,
\]
where the fiber product is taken $n+1$ times, and the simplicial maps are the usual ones. We call $\check{C}\fV_\bullet$ the {\em \v{C}ech nerve} of the cover $\fV$.
\end{definition}

A homotopy sheaf will be a simplicial presheaf that satisfies homotopical descent for every such cover.

\begin{definition}
A simplicial presheaf $\cF$ on a nilpotent dg manifold $\cM$  is a {\em homotopy sheaf} if for every open $U$ of $M$ and every cover $\fV$ of $U$, we have
\[
\cF(U) \xrightarrow{\simeq} \holim_{\check{C}\fV}\,\cF,
\]
where $\simeq$ denotes weak equivalence of simplicial sets and $\holim$ denotes the homotopy limit.\end{definition}

We provide a concise introduction to homotopy limits, with a focus on the case of interest, in appendix \ref{app:holim}. 

\subsection{The road not (yet) taken}

So far we have given the basic skeleton of an approach to geometry, but much remains to be fleshed out. We wish here to point out a few possibilities that we find particularly interesting and then to explain some choices made in \cite{WG2}.

First, we only spell out categories with weak equivalences here, both of nilpotent dg manifolds and of derived stacks. Many constructions would undoubtedly work more easily if one carefully constructed simplicially-enriched or quasi-categories of these objects. (We expect that for doing serious work in this setting, there might be more geometrically natural ways of constructing these $\infty$-categories than taking the Dwyer-Kan localization.)

Second, it would be useful to construct categories of $\sO_X$-modules, quasi-coherent sheaves, and so on over these spaces. In general, there are many techniques and examples in derived algebraic geometry whose analogues would be very useful in our setting. For example, for applications to field theory, it would be nice to have stacks like the moduli of Riemann surfaces or the moduli of holomorphic $G$-bundles on a complex manifold.

Finally, we mention that nilpotent dg manifolds appear already in Costello's approach to quantum field theory, where he shows that renormalization and Feynman diagram computations behave well in families over nilpotent dg manifolds (see section 13 of chapter 2 of \cite{Cos1}). There, he relies crucially on the fact that any constant or linear terms in the action functional are a multiple of the nilpotent ideal. It would be interesting to see if one could modify that analysis to work over dg manifolds whose structure sheaves are cohomologically artinian (and not already artinian at the cochain level), as these test objects are possibly more natural from the derived geometry perspective.

\section{$\L8$ spaces}

In deformation theory, there is a governing, heuristic principle: every deformation functor is given by a dg Lie algebra.\footnote{This idea has a long history, which we will not trace here. See Hinich's paper \cite{Hinich} for one treatment of this idea that is quite close to what we do here. In \cite{LurieDAGX}, Lurie proves a theorem that makes this principle precise and connects it with global derived geometry. Hennion \cite{Hennion} extends Lurie's result to a relative context, working over a base derived Artin stack.} In other words, we can describe the ``formal neighborhood of a point in some space'' using Lie-theoretic constructions, rather than commutative algebra constructions. (One recovers the functions on the formal neighborhood by taking the Chevalley-Eilenberg cochain complex that computes the cohomology of the dg Lie algebra.) Often, this perspective is incredibly helpful, partly because the manipulations on the Lie side may be simpler.

Our primary interest is in families of deformation problems parametrized by smooth manifolds, so we might hope we get a nice kind of derived stack by equipping a smooth manifold with a sheaf of dg Lie algebras. We make this idea precise via the notion of an {\em $\L8$ space}.

\subsection{Curved $\L8$ algebras}

It is convenient to enlarge the Lie-theoretic side to make it more flexible. We will work with curved $\L8$ algebras rather than dg Lie algebras. 

\begin{definition}
Let $A$ be a commutative dg algebra with a nilpotent dg ideal $I$. A {\em curved $\L8$ algebra over $A$} consists of
\begin{enumerate}
\item[(1)] a locally free, $\ZZ$-graded $A^\sharp$-module $V$ and
\item[(2)] a linear map of cohomological degree 1
\[
d: \Sym (V[1]) \to  \Sym (V[1])
\]
\end{enumerate}
satisfying
\begin{enumerate}
\item[(i)] $d^2 = 0$,
\item[(ii)] $(\Sym (V[1]),d)$ is a cocommutative dg coalgebra over $A$ (i.e., $d$ is a coderivation), and
\item[(iii)] modulo $I$, the coderivation $d$ vanishes on the constants (i.e., on $\Sym^0$).
\end{enumerate}
\end{definition}

The notation $\Sym (V[1])$ indicates the graded vector space known as the symmetric algebra over the graded algebra $A^\sharp$ underlying the dg algebra $A$. We only remember its natural coalgebra structure in this setting. To reduce notation, we use $C_\ast(V)$ to denote the cocommutative dg coalgebra $(\Sym (V[1]),d)$. We use this notation because we call it the \emph{Chevally-Eilenberg homology complex} of $V$, as we are extending the usual notions of Lie algebra homology. 

Recall that we obtain a sequence of maps
\[
\ell_n: (\Lambda^n V)[n-2] \to V,
\]
the $n$-fold bracket on $V$, from the composition
\[
\Sym^n(V[1]) \hookrightarrow C_*(V) \overset{d}{\to} C_*(V) \overset{\pi}{\twoheadrightarrow} \Sym^1(V[1]) = V[1],
\]
by shifting by 1. Thinking of $V$ equipped with these brackets is why we use the terminology $\L8$ algebra; it is often easier to work with the Chevalley-Eilenberg homology complex, which assembles all the brackets into a single map.

There is also a natural Chevalley-Eilenberg \emph{cohomology} complex $C^\ast(V)$. It is $(\csym (V^\vee[-1]),d)$, where the notation $\csym (V^\vee[-1])$ indicates the completed symmetric algebra over the graded algebra $A^\sharp$ underlying the dg algebra $A$.  The differential $d$ is the ``dual" differential to that on $C_*(V)$. In particular, it makes $C^*(V)$ into a commutative dg algebra, so $d$ is a derivation.

We usually think of a curved $\L8$ algebra $\fg$ over $A$ as describing a derived space $B\fg$ over $\Spec A$. The algebra of functions of $B\fg$ is precisely its Chevalley-Eilenberg cohomology complex $C^* (\fg)$. Thanks to the natural pairing between the cohomology and homology complexes, we view $C_\ast(\fg)$ as the coalgebra of distributions on $B\fg$.

\begin{definition}
A {\em map of curved $\L8$ algebras} $\phi: V \to W$ is a map of cocommutative dg coalgebras $\phi_*: C_*(V) \to C_*(W)$ respecting the cofiltration by $I$. A map is a {\em weak equivalence} if $\phi_*$ is a quasi-isomorphism.
\end{definition}

\subsubsection{Commentary on curving}

A curious aspect of this definition is the curving, since the uncurved case is discussed far more often. Indeed, ``flat" $\L8$ algebras (i.e., with zero curving) are usually understood as describing pointed formal moduli problems (see, for example, Lurie's ICM talk \cite{LurieICM} for a recent discussion). If $\fg$ is the flat $\L8$ algebra over a commutative dg algebra $R$, the moduli problem $B\fg$ has a marked point. On the commutative algebra side, this appears as the fact that $C^*\fg$ is augmented: there is a distinguished map $C^*\fg \to R$. The derived space $B\fg = \Spec C^*\fg$ is thus pointed by the augmentation map $\Spec R \to \Spec C^*\fg$. A curved $\L8$ algebra $\fg$ then corresponds to a nonpointed formal moduli spaces, because $C^* \fg$ is not augmented. We now elaborate on this idea.

Let $R$ be a commutative dg algebra with nilpotent ideal $I$ and let $S$ denote $R/I$. Given a curved $\L8$ algebra $\tilde{\fg}$ over $R$, let $\fg$ denote the reduction modulo $I$, which is a flat $\L8$ algebra over $S$. Let $B\tilde{\fg}$ denote the space associated to the algebra $C_R^* \tilde{\fg}$, which is a semi-free algebra over $R$, and let $B\fg$ denote the space associated to $C_S^* \fg$, which is a semi-free algebra over $S$. The space $B\tilde{\fg}$ encodes a fattening of the pointed space $B\fg$, where we cannot extend the map $p: \Spec S \to \Spec C_R^* \tilde{\fg}$ to an $R$-point $\tilde{p}: \Spec R \to \Spec C_R^* \tilde{\fg}$. 
\[
\xymatrix{
B\fg \ar[d] \ar[r] & B\tilde{\fg} \ar[d] \\
\Spec S \ar[r] \ar[ur]^p & \Spec R
}
\]
The curving is the obstruction to such an extension.

\begin{remark}
Here is a different way of concocting such a situation. Consider a map of commutative dg algebras $f: A \to B$, which we view as a map of ``derived spaces" $\Spec B \to \Spec A$. This map makes $B$ an $A$-algebra and so we can find a semi-free resolution $\Sym_A(M)$ of $B$ as an $A$-algebra. This replacement $\Sym_A(M)$ expresses $B$ as a kind of $\L8$ algebra over $A$, namely $\fg_B = M^\vee[-1]$.\footnote{This is not strictly true, because we are not working with the completed symmetric algebra, but we're simply providing motivation here.} Note that if $f$ factors through a quotient $A/I$ of $A$, then $\fg_B$ will be curved. (The differential for the semi-free resolution will produce $I$ as the image of the differential's Taylor component mapping to $\Sym^0_A(M) = A$.) This curving appears because $\Spec B$ really only lives over the subscheme $\Spec A/I \subset \Spec A$, and extending it over the rest of $\Spec A$ is obstructed.
\hfill $\Diamond$ \end{remark}

This kind of situation appears in the category of nilpotent dg manifolds. For any $\cM = (M, \sO_\cM)$, we see that $\cM$ ``lives between" the smooth manifold $M_{sm}$ and its de Rham space $M_{dR}$ because we have algebra maps
\[
\Omega^*_M \to \sO_\cM \overset{q}{\to} \cinf_M
\]
by definition. These maps induce maps of nilpotent dg manifolds
\[
M_{sm} \to \cM \to M_{dR},
\]
where the underlying map of manifolds is simply the identity. We will see that we can often find a ``replacement" of $\cM$ as a kind of $\L8$ algebra over $M_{dR}$.

\subsection{The Maurer-Cartan equation}\label{sec:MCeqn}

For $\alpha$ a degree 1 element of a curved $\L8$ algebra $\fg$, let the \emph{Maurer-Cartan element} be 
\[
mc(\alpha) := \sum_{n = 0}^\infty \frac{1}{n!} \ell_n(\alpha^{\ot n}).
\]
The \emph{Maurer-Cartan equation} is then $mc(\alpha) = 0$. There are many useful interpretations of this equation (and we discuss some in the appendix \ref{app:MC}).\footnote{Something that might rightly bother the reader is that the equation involves an infinite sum, which will not be well-defined in most cases. We will only work with \emph{nilpotent} elements $\alpha$, so that the sum is actually finite. Indeed, we use $\fg$ and $mc$ to construct a functor on artinian algebras: tensoring $\fg$ with the maximal ideal of an artinian algebra gives us a nilpotent $\L8$ algebra on which $mc$ is well-behaved.} 

Here we will emphasize that a map of commutative dg algebras $a: C^*(\fg) \to A$ is determined by a cochain map $a: \Sym^1(\fg^\vee[-1]) = \fg^\vee[-1] \to A$, since a map of algebras is determined by where the generators go. Consider the element $\alpha \in \fg^{\vee \vee}[1] \otimes A$ that is dual to $a \in \Hom(\fg^\vee[-1],A)$. Then the condition of $a$ being a cochain map is precisely the Maurer-Cartan equation. on $\alpha$ (under the finiteness condition that $\fg^{\vee \vee} \cong \fg$). In sum, if we view $\fg$ as encoding some kind of space $B\fg$, the Maurer-Cartan equation lets us understand its $A$-points.

We now construct a simplicial set of solutions to the Maurer-Cartan equation. (As explained in appendix \ref{app:MC}, Getzler's paper \cite{Getzler} is a wonderful reference for this Maurer-Cartan functor and much more.)

\begin{definition}
Let $\fg$ be a curved $\L8$ algebra. The \emph{Maurer-Cartan space} $\MC_\bullet(\fg)$ is the simplicial set whose $n$-simplices are solutions to the Maurer-Cartan equation in the curved $\L8$ algebra $\fg \ot \Omega^*(\triangle^n)$.
\end{definition}

This space $\MC_\bullet(\fg)$ has several nice properties when $\fg$ is nilpotent: for instance, it is a Kan complex. See appendix \ref{app:MC} for more discussion and references.

\subsection{$\L8$ spaces}

We now describe the version of ``families of curved $\L8$ algebras parametrized by a smooth manifold" appropriate to our context.

\begin{definition}
Let $X$ be a smooth manifold. 
\begin{enumerate}
\item A {\em curved $\L8$ algebra over $\Omega^*_X$} consists of a $\ZZ$-graded topological\footnote{That is, the fibers are topological vector spaces and the gluing maps are continuous linear maps. In examples, the fibers will be nuclear Fr\'echet spaces, typically smooth sections of a vector bundle on some manifold. For instance, consider a smooth fiber bundle $p: T \to X$ where the fiber is diffeomorphic to some fixed manifold $M$ and consider a relative vector bundle $E$ on $T$, so that we have a vector bundle on each fiber of $T$ over $X$. Then the vector bundle $V$ might be relative sections of this relative vector bundle: to a point $x \in X$,  we associate the vector space of smooth sections of the vector bundle $E_x \to p^{-1}(x) \cong M$.} vector bundle $\pi: V \to X$ and the structure of a curved $\L8$ algebra structure  on its sheaf of smooth sections, denoted $\fg$, where the base algebra is  $\Omega^\ast_X$ with nilpotent ideal $\sI = \Omega^{\geq 1}_X$.
\item An {\em $\L8$ space} is a pair $(X, \fg)$, where $\fg$ is a curved $\L8$ algebra over $\Omega^\ast_X$.
\end{enumerate}
\end{definition}

We now explain how every $\L8$ space defines a derived stack.  Let $B\fg := (X,\fg)$ denote an $\L8$ space. Given a smooth map $f: Y \to X$, we obtain a curved $\L8$ algebra over $\Omega^*Y$ by
\[
f^* \fg := f^{-1}\fg \otimes_{f^{-1} \Omega^*_X} \Omega^*_Y,
\]
where $f^{-1} \fg$ denotes sheaf of smooth sections of the pullback vector bundle $f^{-1} V$.

\begin{definition}
For $B\fg = (X,\fg)$ an $\L8$ space, its {\em functor of points} is the functor
\[
\bB\fg: \dgMan^{op} \to s\!\Sets
\]
for which an $n$-simplex of $\bB\fg(\cM)$ is a pair $(f,\alpha)$: a smooth map $f: M \to X$ and a solution $\alpha$ to the Maurer-Cartan equation in $f^* \fg \ot_{\Omega^*_M} \sI_\cM \ot_\RR \Omega^*(\triangle^n)$.
\end{definition}

The basic idea of the definition is hopefully clear, but we want to remark upon several choices made in this definition. First, note that we use the nilpotent ideal $\sI_\cM$, not the whole algebra $\sO_\cM$. This restriction ensures that we have a nilpotent curved $\L8$ algebra, and hence a Kan complex. It also encodes the idea that we are deforming in the nilpotent direction away from an underlying map. Second, note that we are \emph{not} allowing the smooth map to vary over the $n$-simplices. In other words, $\bB\fg(\cM)$ is the disjoint union 
\[
\bigsqcup_{f \in C^\infty(M,X)} \MC_\bullet(f^* \fg \ot_{\Omega^*_M} \sI_\cM),
\]
where the union is over the \emph{set} of smooth maps $C^\infty(M,X)$. 

\begin{remark}
If we view $C^* \fg$ as the {\em structure sheaf} of the $\L8$ space, then a vertex of $\bB\fg(\cM)$ is a map of the underlying manifolds $f: M \to X$ and a map of commutative dg algebras $f^{-1} C^* \fg \to \sO_\cM$. In other words, it is a map of dg ringed spaces.
\hfill $\Diamond$ \end{remark}

\begin{theorem}\label{L8IsDerived}
The functor $\bB\fg$ associated to an $\L8$ space $B\fg$ defines a derived stack.
\end{theorem}

As the proof of this theorem is lengthy and somewhat technical, we banish it to appendix \ref{app:proof} to maintain our narrative flow.

\begin{remark}
A few interpretive comments are in order.  

First, one can view an $\L8$ space $(X,\fg)$ as a (non-nilpotent) dg manifold whose structure sheaf $C^* \fg$ is a cofibrant commutative dg algebra over $\Omega^\ast_X$. These are particularly well-behaved class of manifolds equipped with sheaves of commutative algebras, although we do not develop that formalism here. Note that every nilpotent dg manifold $(M, \sA)$ has a ``replacement" by an $\L8$ space, precisely by taking a semi-free resolution of its structure sheaf $\sA$ over $\Omega^*_M$. See remark \ref{rem:derivedenhancement} for more discussion of this point.

Second, as previously noted,  there is a well-known correspondence between $\L8$ algebras and formal moduli spaces. An $\L8$ space is a relative version of this idea: we have a family of formal moduli spaces parametrized by a smooth manifold. For a recent and enlightening treatment of an analogous idea in derived algebra geometry, see \cite{Hennion}, which also gives a clear explanation of how such relative formal stacks compare to derived Artin stacks.
\hfill $\Diamond$
\end{remark}

\subsection{Examples}

\subsubsection{The functor of points evaluated on a smooth manifold}

Let $M_{sm} = (M,\cinf_M)$ be a smooth manifold viewed as a nilpotent dg manifold. Note that the nilpotent ideal $\sI_M = 0$ here. Then we have the following simple observation.

\begin{lemma}\label{lem:basecase}
For any $\L8$ space $B\fg = (X,\fg)$, $\bB\fg(M)$ is the discrete simplicial set given by the set $C^\infty(M,X)$ of smooth maps from $M$ to $X$.
\end{lemma}

\begin{proof}
For any smooth map $f: M \to X$, we see that $f^* \fg \ot_{\Omega^*_M} \sI_M = 0$. Hence there is exactly one solution to the Maurer-Cartan equation: zero. 
\end{proof}

\subsubsection{The de Rham space $X_{dR}$}

Let $X$ be a smooth manifold. Consider the {\em zero} vector bundle, equipped with the trivial $\Omega^*_X$ structure. This $\L8$ space $(X, 0)$ has associated structure sheaf $C^* \fg = \Omega^*_X$: we are recovering the de Rham complex itself. Thus, this $\L8$ space provides a derived stack associated to the de Rham space $X_{dR}$. Abusively, we will also denote this derived stack by $X_{dR}$.

\begin{lemma}\label{lem:dR}
For $\cM = (M, \sO_\cM)$ a nilpotent dg manifold, $X_{dR}(\cM)$ is the discrete simplicial set of smooth maps $\Maps(M,X)$.
\end{lemma}

In other words, $X_{dR}(\cM) = X_{dR}(M_{sm})$ for any nilpotent dg manifold: the derived stack only cares about the underlying smooth manifold and not the structure sheaf. Note that this behavior agrees with the definition of the de Rham stack in algebraic geometry.\footnote{In algebraic geometry, the de Rham stack $X_{dR}$ of a stack $X$ is given by the sheafification of the functor $X_{dR}(R) = X(R/Nil(R))$, where $Nil(R)$ denotes the nilradical. In short, the de Rham stack does not ``see'' nilpotent directions, only the underlying reduced scheme. For further discussion, see \cite{GaitsgoryRozenblyum}.}

\begin{proof}
For $f: M \to X$ a smooth map, we note that $f^{-1} 0 = 0$, so that we have the trivial $\L8$ algebra on $M$, no matter the structure sheaf on $\cM$, and so there is only the zero solution to the Maurer-Cartan equation.
\end{proof}

\subsubsection{An $\L8$ space encoding a smooth manifold}

Finally, we turn to our main example of an $\L8$ space: the one that encodes the smooth geometry of a manifold $X$.  More precisely, we have the following existence lemma from \cite{GG}. We include the proof here to illustrate how $\infty$-jet bundles allow us to find a ``Koszul dual" $\L8$ space to an actual manifold.

Recall that the $\infty$-jet bundle $J$ for the trivial bundle is a pro-vector bundle, whose fiber at a point $x \in X$ encodes the ``Taylor series around $x$'' for smooth functions. The bundle $J$ comes equipped with a canonical flat connection, whose horizontal sections are exactly the smooth functions on $X$. We denote the sheaf of smooth sections of $J$ by $\sJ$. The de Rham complex for $\sJ$, whose differential is given by the canonical flat connection, is denoted $dR(\sJ)$, to lessen the profusion of $\Omega$ throughout this paper.

\begin{lemma}\label{lem:tangentL8}
Let $X$ be a smooth manifold. There is a curved $\L8$ algebra $\fg_X$ over $\Omega_X$, with nilpotent ideal $\Omega^{> 0}_X$, such that
\begin{enumerate} 

\item[(a)] $\fg_X \cong \Omega^\sharp_X(T_X[-1])$ as an $\Omega^\sharp_X$ module;

\item[(b)] $dR(\sJ) \cong C^*(\fg_X)$ as commutative $\Omega_X$ algebras;

\item[(c)] The map sending a smooth function to its $\infty$-jet 
\[
C^\infty_X \hookrightarrow dR(\sJ) \cong C^*(\fg_X)
\]
is a quasi-isomorphism of $\Omega_X$-algebras.

\end{enumerate}
\end{lemma}

\begin{proof}
We need to show that we can equip $\csym_{\cinf_X}(T_X^\vee) \ot_{\cinf_X} \Omega_X$ with a degree $1$ derivation $d$ such that $d^2 = 0$ (this is the curved $\L8$ structure) and such that this Chevalley-Eilenberg complex is quasi-isomorphic to $C^\infty_X$ as an $\Omega_X$ module. In this process we will see property (b) explicitly.

We start by working with $D_X$ modules and then use the de Rham functor to translate our constructions to $\Omega_X$ modules. Consider the sheaf $\sJ$ of infinite jets of smooth functions. Observe that there is a natural descending filtration on $\sJ$ by ``order of vanishing.'' To see this explicitly, note that the fiber of $J$ at a point $x$ is isomorphic (after picking local coordinates $x_1, \ldots, x_n$) to $\CC[[x_1, \ldots, x_n]]$, and we can filter this vector space by powers of the ideal $\mathfrak m = (x_1,\ldots,x_n)$. We define $F^k \sJ$ to be those sections of $\sJ$ which live in $\mathfrak m^k$ for every point. This filtration is not preserved by the flat connection, but the connection does send a section in $F^k \sJ$ to a section of $F^{k-1} \sJ \ot_{\cinf_X} \Omega^1_X$.

Observe that $F^1 \sJ / F^2 \sJ \cong \Omega^1_X$, because the first-order jets of a function encode its exterior derivative. Moreover, $F^k \sJ / F^{k+1} \sJ \cong \Sym^k(\Omega^1_x)$ for similar reasons. Pick a splitting of the map $F^1 \sJ \rightarrow \Omega^1_X$ as $\cinf_X$ modules; we denote the splitting by $\sigma$. (Note that there is a contractible space of such splittings, see the following subsection.) By the universal property of the symmetric algebra, we get a map of non unital $\cinf_X$ algebras that is, in fact, an isomorphism 
\[
 \Sym^{>0}_{\cinf_X}(\Omega^1_X) \xrightarrow{\cong} F^1 \sJ.
\]
Now both $\csym_{\cinf_X} (\Omega^1_x)$ and $\sJ$ are augmented $\cinf_X$ algebras with augmentations
\[
p: \csym_{\cinf_X} (\Omega^1_X) \to \Sym^0 = \cinf_X \text{ and } q: \sJ \to \sJ/ F^1 \sJ \cong \cinf_X .
\]
Further, $\Sym^{>0}_{\cinf_X}(\Omega^1_X) = \ker p$ and $F^1 \sJ = \ker q$, so we obtain an isomorphism of $\cinf_X$ algebras
\[
\csym_{\cinf_X} (\Omega^1_X) \xrightarrow{\cong_\sigma} \sJ
\]
by extending the previous isomorphism by the identity on $\Sym^0_{\cinf_X}$ and $\sJ / F^1 \sJ$.  The preceding discussion is just one instance of the equivalence of categories between commutative non unital $A$ algebras and commutative augmented $A$ algebras for $A$ any commutative algebra.

We then equip $\csym(\Omega^1_X)$ with the flat connection for $\sJ$, via the isomorphism, thus making it into a $D_X$ algebra. Applying the de Rham functor $dR$, we get an isomorphism of $\Omega_X$ algebras
\[
 \csym_{\cinf_X} (\Omega^1_X) \ot_{\cinf_X} \Omega_X \xrightarrow{\cong_\sigma} \sJ \otimes_{\cinf_X} \Omega_X .
\]
Recall that the symmetric algebra is compatible with base change, that is
\[
\csym_{\cinf_X} (\Omega^1_X ) \otimes_{\cinf_X} \Omega^\sharp_X = \csym_{\Omega^\sharp_X} (\Omega^1_X \otimes_{\cinf_X} \Omega^\sharp_X ) \cong \csym_{\Omega^\sharp_X} \left ( (T_X [-1] \otimes_{\cinf_X} \Omega^\sharp_X)^\vee [-1]\right ),
\]
where we dualize over $\Omega^\sharp_X$. Via the de Rham functor we have constructed a derivation on this completed symmetric algebra defining the $\L8$ structure over $\Omega_X$. 

Property (c) follows from a standard argument; see \cite{CFT} for an explicit contracting homotopy. 
\end{proof}

Note that we pick a splitting $\sigma$ in the proof but that the space of splittings is contractible, and that all the associated $\L8$ algebras are strictly isomorphic. We thus make no fuss over the choice of $\sigma$ and denote the resulting $\L8$ space $(X, \fg_X)$ by $B\fg_X$. 

\begin{remark}\label{rem:derivedenhancement}
We view $B\fg_X$ as a natural {\em derived enhancement} of the smooth manifold $X$ for the following reason. From the functor of points perspective, any sheaf of sets on the site $\Man$ of smooth manifolds
\[
\cM: \Man^{op} \to \Sets
\]
is a kind of ``generalized smooth manifold." The representable functor $\underline{X} = \Man(-,X)$ is such a sheaf. Similarly, a homotopy sheaf of simplicial sets $\cM$ on $\Man$ is then a smooth {\em stack}.\footnote{A generalized smooth manifold $\cM$ defines a smooth stack by taking the discrete simplicial set $\cM(Y)$ on every manifold $Y$. The argument parallels lemma \ref{lem:baseofcech}.} A {\em derived enhancement} of a smooth stack $\cM$ is a derived stack 
\[
\widetilde{\cM}: \dgMan^{op} \to s\!\Sets
\] 
such that the restriction to the subsite $\Man \subset \dgMan$ agrees with $\cM$.\footnote{This perspective is standard in the setting of derived geometry. See for instance \cite{STV} or \cite{ToenSeattle}.} By lemma \ref{lem:basecase}, (the derived stack of) any $\L8$ space $(X,\fg)$ is a derived enhancement of $\underline{X}$, since the restriction to $\Man$ does not care about $\fg$. But  the lemma \ref{lem:tangentL8} above shows that $B\fg_X$ essentially provides a ``cofibrant replacement" for the smooth manifold $X$: we have replaced the structure sheaf $C^\infty_X$ with a semi-free resolution over $\Omega_X$. In this sense, it is the most natural derived enhancement.\hfill $\Diamond$ 
\end{remark}

\subsubsection{An $\L8$ space encoding a complex manifold}

The construction of the $\L8$ space $B\fg_X$ is inspired by Costello's work in the holomorphic setting.  If $Y$ is a complex manifold, then there exists an $\L8$ space $Y_{\overline{\partial}} = (Y, \fg_{Y_{\overline{\partial}}})$ with the following properties.

\begin{prop}[Lemma 3.1.1 of \cite{Cos2}]
Let $Y$ be a complex manifold. The $\L8$ space $Y_{\overline{\partial}}$ is well defined up to contractible choice.  Further,
\begin{enumerate}
\item[(a)] as an $\Omega_Y^\sharp$-module, $\fg_{Y_{\overline{\partial}}}$ is isomorphic to $\Omega^\sharp_Y(T^{1,0}_Y [-1])$;
\item[(b)] the derived stack $\bB \fg_{Y_{\overline{\partial}}}$ represents the moduli problem of holomorphic maps into $Y$: for any complex manifold $X$ (viewed as a nilpotent dg manifold), $\bB \fg_{Y_{\overline{\partial}}}(X)$ is the discrete simplicial set of holomorphic maps from $X$ to $Y$.
\end{enumerate}
\end{prop}

\section{Geometry with $\L8$ spaces}

Especially important for us will be that we can thus define shifted symplectic structures on $\L8$ spaces, which play a crucial role in the classical Batalin-Vilkovisky formalism.

\subsection{Vector bundles on $\L8$ spaces}

The notion of $\L8$ space is sufficiently geometric to admit notions of vector bundles, in particular, (co)tangent bundles. 

\begin{definition}
Let $B\fg := (X, \fg)$ be an $\L8$ space.  A {\it vector bundle} on $B\fg$ is a $\ZZ$-graded topological vector bundle $\pi:V \to X$ for which its sheaf of smooth sections $\cV$ over $X$ has the structure of an $\Omega^\sharp_X$-module and for which $\fg \oplus \cV$ has the structure of a curved $\L8$ algebra over $\Omega^*_X$ such that
\begin{enumerate}
\item[(1)] the inclusion $\fg \hookrightarrow  \fg \oplus \cV$ and the projection $\fg \oplus \cV \to \fg$ are maps of $\L8$ algebras, and
\item[(2)] the Taylor coefficients $\ell_n$ of the $\L8$ structure vanish on tensors containing two or more sections of $\cV$.
\end{enumerate}
The {\it sheaf of sections of $\cV$ over $B\fg$} is given by the sheaf on $X$ of dg $C^\ast(\fg)$-modules $C^\ast (\fg, \cV[1])$, the Chevalley-Eilenberg complex for an $\fg$-module. The {\em total space} for the vector bundle $\cV$ over $B\fg$ is the $\L8$ space $(X, \fg \oplus \cV)$.
\end{definition}

In particular, for $X$ a point, we recover the usual notion of a representation of $\fg$. Note that we have merely picked out a class of well-behaved sheaves of $\fg$-modules. 

We now pick out two important examples. Recall that for a semi-free commutative dg algebra $\cA = (\csym(V), d)$, the {\em derivations} are the module $\Der \cA := (\csym(V) \ot V^\vee, d)$. We view the derivations as the vector fields --- the sections of the tangent bundle $T_\cA$ --- on the space corresponding to $\cA$. By this correspondence, we obtain the following.

\begin{definition}
The {\em tangent bundle} $T_{B\fg}$ is given by $\fg[1]$ equipped with the (shifted) adjoint action of $\fg$. Likewise, the {\em cotangent bundle} $T^*_{B\fg}$ is given by $\fg^\vee[-1]$ equipped with the (shifted) coadjoint action of $\fg$.
\end{definition}

There are also {\em shifted} (co)tangent bundles. For instance, we let $T[k]B\fg$ denote the $\L8$ space $(X,\fg \oplus \fg[k+1])$, which is the total space of the vector bundle $T_{B\fg}[k]$.

Sections of $T^*_{B\fg}$ over $B\fg$ are the K\"ahler differentials $\Omega^1_{B\fg}$ of $\sO_{B\fg} = C^* \fg$. They are $\csym(\fg^\vee [-1]) \ot_k \left( \fg^\vee [-1]\right)$ equipped with the differential 
\[
d_{\Omega^1}: f \ot x \mapsto d_\fg f \ot x + (-1)^{|f|} f \cdot d_{dR} (d_\fg x), 
\]
where $f \in \sO_{B\fg}$ and $x \in \fg^\vee[-1]$. Here $d_\fg$ denotes the differential on $C^*\fg$ and $d_{dR}: \sO_{B\fg} \to \Omega^1_{B\fg}$ denotes the universal derivation
\[
d_{dR}: x \mapsto 1 \ot x
\]
for $x \in \fg^\vee[-1]$. Note that $d_\Omega \circ d_{dR} = d_{dR} \circ \partial$. From hereon, we will denote $1 \ot x$ by $dx$ and $f \ot x$ by $f \, dx$.

There is a natural analog of the de Rham complex $DR_{B\fg}$ where $\Omega^k_{B\fg} := C^*(\fg, \wedge^k \fg [-k])$ and 
\[
DR_{B\fg} := \bigoplus_{k \geq 0} \Omega^k_{B\fg}[-k]
\]
with total differential $d_{DR}$ given by adding the ``internal" differential of $\Omega^k$ as a $\fg$-module to the ``exterior derivative" $d_{dR}$ (the universal derivation described in the preceding paragraph).

\subsection{Shifted symplectic structures}

We now formulate the analog of a symplectic form $\omega$ on an $\L8$ space $B\fg = (X,\fg)$. Recall that on a smooth manifold $M$, a 2-form $\omega \in \Omega^2(M)$ is {\em symplectic} if $\omega$ is closed and provides an isomorphism of vector bundles $\underline{\omega}: T_M \to T^*_M$. We must provide derived versions of these two conditions. (For a more sophisticated and thorough treatment of derived symplectic geometry, see \cite{PTVV}. Our definitions are essentially $\L8$ space versions of theirs.)\footnote{As in the case of the Lie-theoretic perspective on deformation theory, there is a long history to these ideas. Our approach grows directly out of the work of Schwarz and Kontsevich and their various collaborators (see, for instance, \cite{Schwarz}, \cite{MovSchw}, and \cite{AKSZ}).}

As described above, we have an analog of the de Rham complex. A {\em closed 2-form} is then a cocycle in the truncated de Rham complex
\[
\Omega^2_{cl} := \left(\bigoplus_{k \geq 2} \Omega^k_{B\fg}[-k+2], d_{DR}\right),
\]
shifted to put $\Omega^2$ ``in degree 0." Observe that there is a natural map
\[
i: \Omega^2_{cl} \to \Omega^2_{B\fg}
\]
by forgetting components living in $\Omega^{> 2}$. Hence being closed is {\em data}, not a {\em property}: given a 2-form $\omega \in \Omega^2_{B\fg}$, one must find a lift to some cocycle $\widetilde{\omega} \in\Omega^2_{cl}$ to have a closed 2-form.

As usual, every element of the $\fg$-module $\cV^\vee \ot \cW$ determines an element of the module $\Hom_{\fg}(\cV,\cW)$. Given a 2-form $\omega$ of cohomological degree $n$, we thus obtain a bundle morphism $\underline{\omega}: T_{B\fg} \to T^*_{B\fg}$ of cohomological degree $-n$. We say $\omega$ is {\em nondegenerate} if the morphism $\underline{\omega}$ is a quasi-isomorphism from $T_{B\fg}$ to $T^*_{B\fg}[-n]$.

\begin{definition}
An {\em $n$-shifted symplectic $\L8$ space} is an $\L8$ space $B\fg = (X,\fg)$ equipped with a closed 2-form $\omega \in \Omega^2_{cl}$ of cohomological degree $-n$ such that $i(\omega)$ is nondegenerate.
\end{definition}

The shifted cotangent bundle $T^*[n]B\fg$ is equipped with a natural $n$-shifted symplectic structure by antisymmetrizing the (shifted) evaluation pairing. 

\section{The derived loop space}

We now discuss an interesting collection of examples --- derived loop spaces --- whose analogs in derived algebraic geometry provide new perspectives on important constructions in representation theory and homological algebra (see, for instance, \cite{BZN} and \cite{TVonLX}). In our work \cite{GG}, their geometry describes the classical field theory that we quantize.

Recall the usual free loop space $LX$, where
\[
LX = \Man(S^1 , X) = \{\text{smooth maps $f: S^1 \to X$}\}.
\]
There are various ways to view $LX$ as a topological space or even an $\infty$-dimensional smooth manifold. For us, the functor of points approach to geometry suggests that we view $LX$ as the functor 
\[
\begin{array}{cccc}
L_{sm}X: & \Man^{op} & \to & \Sets \\[1ex]
& M & \mapsto & \Man(S^1 \times M , X)
\end{array}
\]
where the right hand side is simply the {\em set} of smooth maps $S^1 \times M \to X$. Note that this functor is, in fact, a sheaf on the site $\Man$.

When we move to the derived setting, we can ask for a derived enhancement of $L_{sm}X$, but we can also ask about other kinds of loop spaces using other circles. In particular, two other ``circles" besides $S^1$ are: 
\begin{enumerate}
\item[(1)] the de Rham circle $S^1_{dR} = (S^1, \Omega^*_{S^1})$, and
\item[(2)] the ``Betti circle" $S^1_\cB = (\pt, \RR[\epsilon])$ with $\epsilon$ of degree 1 (i.e., the functions are the cohomology ring of the circle).
\end{enumerate}
The de Rham circle is a nilpotent dg manifold but the Betti circle is not, strictly speaking, because the algebra does not sit in nonpositive degrees. Nonetheless, it is a reasonable ringed space to consider. In essence, $S^1_{dR}$ knows about the smooth topology of $S^1$ and $S^1_\cB$ knows about the real homotopy type of $S^1$.

We now describe derived stacks that are natural loop spaces for these various circles. Note that we will always be discussing some version of a loop space for a manifold $X$, but our constructions work for an arbitrary $\L8$ space by replacing $B\fg_X$ by an arbitrary $B\fg$.

\subsection{Enhancing $L_{sm}X$}

Lemma \ref{lem:tangentL8} gives us a natural derived enhancement of $X$ itself: use the $\L8$ space $B\fg_X$. We piggyback on that result to enhance $L_{sm} X$.

\begin{definition}
Let $\cL_{sm} X$ denote the functor 
\[
\begin{array}{cccc}
\cL_{sm}X: & \dgMan^{op} & \to & s\!\Sets \\[1ex]
& \cM & \mapsto &  \bB\fg_X(S^1 \times \cM)
\end{array}.
\]
\end{definition}

This functor is, in fact, a derived stack because $\bB\fg_X$ is.

\subsection{The Betti loop space $\cL_{\cB}X$}

Let $X$ be a smooth manifold and let $B\fg_X = (X, \fg_X)$ be the associated $\L8$ space. Let $\cL_{\cB} X$ denote the $\L8$ space $(X, \RR[\epsilon] \ot \fg_X)$. We call it the {\em Betti loop space} of $X$.

\begin{remark}
If $S^1_\cB$ were a nilpotent dg manifold, the functor $\bB\cL_\cB X$ would agree with $\bB\fg_X(S^1_\cB \times -)$. Hence this definition is motivated by the same logic as $L_{sm}X$ or $\cL_{sm}X$.
\hfill $\Diamond$ \end{remark}

The Betti loop space has another description: it is just the $\L8$ space $T[-1]B\fg_X$.\footnote{This statement is the manifestation, in our formalism, of the fact that the derived loop space of a scheme $X$ is $T[-1]X$. See \cite{BZN} for some discussion of this fact.} Explicitly, this $\L8$ space is
\[
(X, \fg_X \oplus \fg_X[-1]),
\]
where $\fg_X[-1]$ denotes $\Omega^*_X(T^*_X[-2])$ viewed as an $\fg_X$-module by the shifted adjoint action. 

Observe that there is a natural quasi-isomorphism
\[
\bigoplus_{k \geq 0} \Omega^k_X[k] \overset{\simeq}{\hookrightarrow} C^*(\fg_X \oplus \fg_X[-1]),
\]
extending the quasi-isomorphism from lemma \ref{lem:tangentL8}. This equivalence says that ``the functions on the Betti loop space are the (regraded) de Rham forms of $X$," that is, $\sO(\cL_\cB X) \simeq \bigoplus_{k \geq 0} \Omega^k_X[k]$. This description meshes nicely with the Hochschild-Kostant-Rosenberg theorem and its interpretation in derived geometry: the Betti loop space encodes the derived self-intersection of $X$ as the diagonal in $X \times X$, and hence its algebra of functions should be the Hochschild homology of functions on $X$. (See \cite{TVonHKR} and \cite{BZN} for a lot more on this topic.)

\begin{remark}
The Betti loop space, and the de Rham loop space defined below, only become interesting when evaluated on interesting nilpotent dg manifolds. When restricted to the subcategory $\Man$ of ordinary manifolds, these functors are isomorphic to the underlying smooth manifold $X$.\hfill $\Diamond$ 
\end{remark}

\subsection{The de Rham loop space $\cL_{dR} X$ and its completion along constant maps}

Following the pattern so far, we define the {\em de Rham loop space} of $X$ as follows.

\begin{definition}
Let $\cL_{dR} X$ denote the functor 
\[
\begin{array}{cccc}
\cL_{dR}X: & \dgMan^{op} & \to & s\!\Sets \\[1ex]
& \cM & \mapsto &  \bB\fg_X(S^1_{dR} \times \cM)
\end{array}.
\]
\end{definition}

This functor is again a derived stack. 

Our main interest in \cite{GG}, however, lies with a simpler space: we focus on maps out of $S^1_{dR} \times \cM$ where the underlying smooth map $S^1 \times M \to X$ is constant along $S^1$. Thanks to a result of Costello, this subfunctor is, in fact, given by an $\L8$ space. We start with his general result before discussing its use in our context.

Let $B\fg =(X, \fg)$ be an $\L8$ space and let $\cN = (N , \sO_\cN)$ be a nilpotent dg manifold.  We define a new simplicial presheaf $\bB\fg^\cN$ on the site $\dgMan$ by
\[
\begin{array}{cccc}
\bB\fg^\cN: & \dgMan^{op} & \to & s\!\Sets \\[1ex]
& \cM & \mapsto & \bB\fg (\cN \times \cM)
\end{array}.
\]
Consider the sub-simplicial presheaf $\widehat{\bB\fg^\cN} \subset \bB\fg^\cN$ given by Maurer-Cartan solutions where the underlying map of manifolds $N \times M \to X$ is independent of $N$.  That is, for $\cM$ a nilpotent dg manifold
\[
\widehat{\bB\fg^\cN}(\cM) \subset \bB\fg (\cN \times \cM)
\]
consists of Maurer-Cartan elements $(f, \alpha)$ such that the underlying smooth map $f$ factors through the projection onto $M$:
\[
\xymatrix{
N \times M \ar[d]_{\pi_M} \ar[r]^f & X \\
M \ar@{-->}[-1,1]
}.
\]
Costello shows that this subfunctor is itself an $\L8$ space, under certain conditions.

Every nilpotent dg manifold $\cN = (N, \sO_\cN)$ has a natural filtration $F^k \sO_\cN = \sI_\cN^k$ by powers of its nilpotent ideal $\sI_\cN$. Let $\Gr^k \sO_\cN = F^k \sO_\cN/F^{k+1} \sO_\cN$ denote the $k$th  component of the associated graded sheaf of algebras.

\begin{prop}[Proposition 5.0.1 of \cite{Cos2}]\label{prop:CostelloRep}
Let $\cN = (N, \sO_\cN)$ be a nilpotent dg manifold such that the cohomology of $\Gr^k \sO_\cN$ is concentrated in strictly positive degrees for $k \geq 1$.
Further, let $B\fg = (X,\fg)$ be an $\L8$ space such that the cohomology of the sheaf of $\L8$ algebras
\[
\fg^{red} := \fg/\Omega^{>0}_X
\]
is concentrated in strictly positive degrees.
Then, the simplicial presheaf $\widehat{\bB\fg^\cN}$ is  weakly equivalent to the functor of points for the $\L8$ space $\fh := (X, \fg \otimes_\RR \sO_\cN(N))$, where $\sO_\cN(N)$ denotes the global sections of the structure sheaf $\sO_\cN$. In other words, we have a natural transformation of functors
\[
\bB\fh \xrightarrow{\simeq} \widehat{\bB\fg^\cN}
\]
that is an objectwise weak equivalence.
\end{prop}

Let us apply this result to the case of interest, where the $\L8$ space is $B\fg_X = (X, \fg_X)$ and the nilpotent dg manifold is $S^1_{dR} = (S^1 , \Omega_{S^1})$. Then $\cL_{dR} X$ is precisely the derived space $B\fg_X^{S^1_{dR}}$. Moreover, our spaces satisfy the hypotheses of the proposition, since the sheaves $\Gr^k (\Omega^*_{S^1}) = \Omega^{k}_{S^1}$ are soft (and hence have vanishing higher cohomologies) and since $\fg_X^{red}$ is already concentrated in positive degrees. Hence, we have the following.

\begin{deflem}
The $\L8$ space
\[
\widehat{\cL_{dR}X} := (X, \Omega^*(S^1) \ot \fg_X)
\]
encodes the derived stack $\widehat{B\fg_X^{S^1_{dR}}}$.
\end{deflem}

Finally, we note the following, which makes this space particularly easy to understand.

\begin{lemma}
The $\L8$ spaces $\cL_\cB X$ and $\widehat{\cL_{dR}X}$ determine weakly equivalent derived stacks. In fact, any volume form $\omega$ on $S^1$ induces a natural transformation $\underline{\omega}: \cL_\cB X \Rightarrow \widehat{\cL_{dR}X}$ that is a homotopy equivalence on every nilpotent dg manifold.
\end{lemma}

\begin{proof}
The first statement directly follows from the second. Note that a choice of $\omega$ gives a quasi-isomorphism of commutative dg algebras 
\[
\begin{array}{cccc}
\phi_\omega: &\RR[\epsilon] & \to & \Omega^*(S^1) \\
&a+ b\epsilon & \mapsto & a + b \omega
\end{array}.
\]
Hence we obtain a map of $\L8$ spaces 
\[
(id_X, \phi_\omega): \cL_\cB X = (X, \fg_X \ot \RR[\epsilon]) \to (X, \fg_X \ot \Omega^*(S^1)) = \widehat{\cL_{dR}X}
\]
that is a quasi-isomorphism of sheaves of curved $\L8$ algebras. This map induces the desired natural transformation between their Maurer-Cartan functors. In particular, given a nilpotent dg manifold $\cM = (M,\sO_\cM)$, for each smooth map $f: M \to X$, we obtain a quasi-isomorphism of nilpotent curved $\L8$ algebras 
\[
f^* \phi_\omega: f^* (\fg_X \ot \RR[\epsilon]) \ot_{\Omega^*_M} \sI_\cM \to f^* (\fg_X \ot \Omega^*(S^1)) \ot_{\Omega^*_M} \sI_\cM
\] 
and hence a homotopy equivalence of Maurer-Cartan spaces (see appendix \ref{app:MC}).
\end{proof}

\subsection{Symplectic structures and the AKSZ construction}

We now explain how to equip $\widehat{\cL_{dR}X}$ with a $-1$-symplectic structure when $X$ is $0$-symplectic. Our approach is a version of the AKSZ construction of Alexandrov, Kontsevich, Schwarz, and Zaboronsky \cite{AKSZ}, who introduced a general method for constructing shifted symplectic structures on mapping spaces. Their motivation was to produce ``classical Batalin-Vilkovisky theories," where the BV (Batalin-Vilkovisky) formalism is a homological approach to field theory. The language of $\L8$ spaces is well-suited to applying the AKSZ approach, as we will demonstrate in the case of the derived loop space $\widehat{\cL_{dR}X}$. Indeed, in \cite{GG}, we start with this space and then quantize it using the BV formalism, as developed in \cite{Cos1}. 

Before equipping the derived loop space with a $-1$-symplectic structure, we provide a quick gloss of the AKSZ construction as motivation. The basic idea is simple. Let the $d$-dimensional source dg manifold $\Sigma$ of a sigma model come equipped with a volume form $\dvol$ and let the target dg manifold $X$ come equipped with a $k$-symplectic structure $\omega$. Then the mapping space ({\it aka} fields) $\cF := \Maps(\Sigma,X)$ obtains a $k-d$-symplectic structure as follows. For a fixed map $\phi: \Sigma \to X$, the tangent space $T_\phi \cF = \Gamma(\Sigma, \phi^* T_X)$ has a natural pairing 
\[
\langle \zeta, \zeta' \rangle_\phi :=  \int_{s \in \Sigma} \phi^*\omega(s) (\zeta(s), \zeta'(s)) \dvol,
\] 
with $\zeta, \zeta \in T_\phi \cF$; and $\langle -,-\rangle_\phi$ has cohomological degree $n-d$ by construction and is skew-symmetric thanks to the skew-symmetry of $\omega$. In many situations, it will be closed and nondegenerate, as well.

This construction can be realized in our context. Suppose $X$ is a $0$-symplectic manifold (i.e., a symplectic manifold in the standard sense) with symplectic form $\omega \in \Omega^2(X)$. Suppose we fix a 1-form $\nu \in \Omega^1(S^1)$ that is not exact, i.e., $\int_{S^1} \nu \neq 0$. 

By corollary 11.3 of \cite{GG}, we know that the $\infty$-jet $J(\omega)$ is a cocycle in both $\Omega^2(B\fg_X)$ and $\Omega^2_{cl}(B\fg_X)$. (This is a special feature of this situation.) Consider the pairing
\[
\begin{array}{cccc}
\Omega_{\omega,\nu}:&[\fg_X \ot \Omega^*(S^1)]^{\ot 2} &\to &\Omega^*(X)\\
&(Z \ot \alpha) \ot (Z' \ot \alpha') &\mapsto &{\displaystyle \int_{\theta \in S^1} }J(\omega)(Z \ot \alpha(\theta), Z' \ot \alpha'(\theta)) \wedge \nu(\theta)
\end{array},
\]
which is a direct application of the AKSZ approach. Note that it is a cocycle in $\Omega^2(\fg_X \ot \Omega^*(S^1))$ by construction.

\begin{lemma}
The 2-form $\Omega_{\omega, \nu}$ is a -1-symplectic form on $\widehat{\cL_{dR}X}$.
\end{lemma}

\begin{proof}
Checking that it is a closed 2-form is a direct computation. It remains to show that the pairing is nondegenerate. Observe that $J(\omega)$ is a nondegenerate pairing on $T_{B\fg_X}$, thanks to corollary 11.3 of \cite{GG}. (Indeed, the horizontal sections of jets of vector fields on $X$ is precisely the sheaf of smooth vector fields on $X$.) As $\nu$ is cohomological nontrivial on $S^1$, we see that $\Omega_{\omega,\nu}$ is nontrivial at the level of cohomology as well.
\end{proof}

Similar arguments should work for many $\L8$ spaces of this form (i.e., an $\L8$ space that arises as completion of a mapping stack along some set of maps). Thus, this formalism is a natural place to deploy the AKSZ construction.


\appendix

\section{Homotopy limits and cosimplicial simplicial sets}\label{app:holim} 

In a category with a notion of weak equivalence (such as topological spaces with weak homotopy equivalence or chain complexes with quasi-isomorphism), the {\em homotopy} limit of a diagram is better behaved than the ordinary limit with regards to questions about objects up to weak equivalence. There is an extensive literature that motivates, defines, and constructs homotopy limits in a variety of contexts, and we recommend \cite{Dugger} and \cite{Shulman} as nice places to start reading. (For a more succinct discussion close to the style of our overview, see \cite{DouglasChap}.)

In our context, we can take advantage of the existence of several explicit, well-behaved formulas for the homotopy limit of a diagram in simplicial sets.\footnote{We emphasize that a homotopy limit should be defined by a homotopic version of the usual universal property for a limit. These formulas provide an explicit means for constructing an object that provides such a homotopy limit.} A standard source is the foundational work of Bousfield and Kan \cite{BK} (a recent, thorough source is \cite{Hirschhorn}). Of course, it is not always obvious how the different formulas are related. We give a quick discussion of the ideas behind the formulas we use.

To motivate all this formalism, we remind the reader of its appearance in this paper. Let $\cF$ be a simplicial presheaf and $\fV = \{V_i \; : \; i \in I \}$ a cover of $U$.  Applying the presheaf $\cF$ levelwise to the \v{C}ech nerve $\check{C}\fV_\bullet$, we obtain a cosimplicial simplicial set
\[
\check{C}^\bullet(\fV; \cF) := 
\xymatrix{
{\prod_{i_{0}} \cF ( V_{i_{0} \in I}})
  \ar@<0.7ex>[r]^-{d^{0}}
  \ar@<-0.7ex>[r]_-{d^{1}} 
&
 {\prod_{i_{0}, i_{1} \in I} \cF (V_{i_{0}i_{1}}})
  \ar@<1ex>[r]^-{d^{0}}
  \ar[r]|-{d^{1}}
  \ar@<-1ex>[r]_-{d^{2}} 
&
 {\prod_{i_{0}, i_{1}, i_{2} \in I} \cF ( V_{i_{0}i_{1}i_{2}}})
  \ar@<2ex>[r]^-{d^{0}}
  \ar@<.66ex>[r]|-{d^{1}}
  \ar@<-.66ex>[r]|-{d^{2}} 
  \ar@<-2ex>[r]_-{d^{3}}
&
{\ldots}
},
\]
called the {\it \v{C}ech cosimplicial diagram} associated to $\fV$ with coefficients in $\cF$.  The standard notation is
\[
\check{C}^p (\fV ; \cF) := \prod_{i_0 \cdots i_p} \cF (V_{i_0 \dotsb i_p}),
\]
where $V_{i_0 \dotsb i_p} = V_{i_0} \cap \cdots \cap V_{i_p}$. We want to compute 
\[
\holim_{\Delta} \check{C}^\bullet (\fV; \cF),
\]
so we need to define a homotopy limit over a cosimplicial diagram. 

\subsection{A quick overview of the big picture.}

Let $I$ denote a small category, and let $\Fun(I,\cC)$ denote the category of functors from this index category $I$ to $\cC$. We call such a functor $X$ an $I$-{\em diagram} in $\cC$ and so call $\Fun(I,\cC)$ a {\em diagram category}. Recall that the limit of a diagram $X$ arises from an adjunction
\[
\C: \cC \leftrightarrows \Fun(I,\cC): \lim_I
\]
where for $x \in \cC$, $\C x$ denotes the diagram 
\[
i \mapsto x \text{ and } (i \to i') \mapsto (x \overset{1_x}{\to} x),
\]
{\em aka} the constant diagram with value $x$. It is straightforward to unravel this definition to the definition by the terminal cone.

Now suppose $\cC$ is a category with a notion of weak equivalence. Let $\Ho(\cC)$ denote the homotopy category given by localizing at these weak equivalences. We equip $\Fun(I,\cC)$ with an associated notion of weak equivalence: a map of diagrams $f:X \to Y$ is a weakly equivalence if the map on each object $f(i): X(i) \to Y(i)$ is a weak equivalence. Let $\Ho(\Fun(I,\cC))$ denote the associated homotopy category.

Observe that the constant functor $\C$ preserves weak equivalences and hence induces a functor $\Ho(\C): \Ho(\cC) \to \Ho(\Fun(I,\cC))$. We denote the right adjoint of $\Ho(\C)$ --- if it exists --- by $\Ho-lim_I$.\footnote{We use this ugly notation to avoid conflicting with commonly-used notations.} {\it Note that $\Ho-lim_I$ does not arise from $\lim_I$ --- it is not ``$\Ho(\lim_I)$" --- except in very special situations, because $\lim_I$ need not preserve weak equivalences.}

Not only do we want to construct $\Ho-lim_I$, but we want to go a step further and produce a functor $\holim_I: \Fun(I,\cC) \to \cC$ that induces $\Ho-lim_I$ at the level of homotopy categories.

Model categories are a well-established approach (among others) to accomplishing such goals.\footnote{Our discussion only assumes a basic familiarity with model categories. For more extensive discussion, there are many lovely expositions, such as \cite{DwyerSpalinski} or \cite{Hovey}.} One useful feature of model categories is that they provide another tool for constructing functors between homotopy categories: given a Quillen adjunction
\[
L: \cC \leftrightarrows \cD: R
\]
between model categories, one obtains an adjunction
\[
\LL L : \Ho(\cC) \leftrightarrows \Ho \cD: \RR R
\]
between the homotopy categories. We now use this tool to find our desired functor $\holim_I$, following Bousfield and Kan.

We are only interested here in the case where $\cC$ is the category $s\!\Sets$ of simplicial sets, with the standard {\it aka} Quillen model structure. Bousfield and Kan introduced a model structure on the diagram category $\Fun(I,\ssets)$, now known as the {\em projective} model structure, where weak equivalences and fibrations are both objectwise.

The trick to making $\holim_I$ is to find a functor $\C': \ssets \to \Fun(I,\ssets)$ such that
\begin{enumerate}
\item[(1)] there is a Quillen adjunction $\C': \ssets \leftrightarrows \Fun(I,\ssets): F$, and 
\item[(2)] $\C'$ is {\em not} $\C$ but $\LL \C' = \Ho(\C)$, so they agree at the level of homotopy categories.
\end{enumerate}
We then know that $\RR F$ is $\Ho-lim_I$, so we can view the right adjoint $F$ as the desired $\holim_I$. In particular, on an objectwise-fibrant diagram $X$, we know that $\holim_I X$ is a simplicial set with the correct homotopy type.

We now construct $\C'$. For any $i \in I$, we denote its over-category by $I_{/i}$. Let $N(I_{/i})$ denote the nerve of $I_{/i}$, which is a contractible simplicial set. We define a functor
\[
\begin{array}{cccc}
\C' = I_{/-} \times -: & \ssets & \to & \Fun(I,\ssets) \\
 & x & \mapsto & (i \mapsto x \times N(I_{/i}))
\end{array}.
\]
By construction, the diagram $\C'(x)$ assigns to every $i \in I$ a simplicial set $\C'(x)(i) = x \times N(I_{/i})$ that is weakly equivalent to $x$. Indeed, there is a canonical natural transformation $\eta: \C \Rightarrow \C'$ because $N(I_{/i})$ has a canonical basepoint $*$ coming from the terminal object $i \to i$ in $I_{/i}$ and so $\eta(x)$ is the canonical map $x \hookrightarrow x \times * \subset x \times N(I_{/i})$. Then $\eta$ is an objectwise weak equivalence. 

Recall that $\ssets$ is enriched over itself, and there is an inner hom that we denote $\Maps$, so 
\[
\ssets(x \times y, z) \cong \ssets(x, \Maps(y,z)). 
\]
It is immediate from this adjunction that the set of $n$-simplices of $\Maps(y,z)$ is $\ssets(\triangle[n] \times y,z)$. Piggybacking on this construction, we obtain a right adjoint to the functor $\C'$:
\[
\begin{array}{cccc}
\Maps(I_{/-}, -): & \Fun(I,\ssets) & \to & \ssets \\
 & X & \mapsto & ([n] \mapsto Nat(I_{/-} \times \C'\triangle[n],X),
\end{array}
\]
where $Nat(A,B)$ denotes the set of natural transformations between the diagrams $A$ and $B$.

\begin{prop}\label{prop:universalprop}
The adjunction 
\[
\C' = I_{/-} \times -: \ssets \leftrightarrows \Fun(I,\ssets): \Maps(I_{/-},-) 
\]
is a Quillen adjunction, using the projective model structure on the right hand side.
\end{prop}

We thus define the {\em homotopy limit} $\holim_I$ to be $\Maps(I_{/-},-)$. Its right derived functor is $\Ho-lim_I$.

Our main interest is when $I = \Delta$, so that $\Fun(I,\ssets)$ is precisely $\cssets$, the cosimplicial simplicial sets.

\begin{definition}
The \emph{homotopy limit} of a cosimplicial diagram of simplicial sets $X \in \cssets$ is 
\[
\holim_\Delta X = \Maps(\Delta_{/-},X).
\]
\end{definition}

We emphasize that this is not the only definition used in the literature, but any construction/definition should provide a weakly equivalent simplicial set (so long as we agree on the notion of weak equivalence of diagrams).

\subsection{An alternative approach: fat totalization}

When $I = \Delta$, one can give a different functor, known as {\em fat totalization} and denoted $\ftot$, whose right derived functor $\RR \ftot$ is $\Ho-lim_\Delta$. In other words, we will recall a different approach to the homotopy limit which has its foundations in the work of Segal \cite{SegalCohom}.

\begin{remark}
Costello uses this version of $\holim$ in his proof, which is why we include this brief discussion.
\end{remark}

Let $\Delta^{inj}$ denote the subcategory of $\Delta$ with the same objects but whose morphisms are just the injections. The superscript $inj$ is to indicate that it is generated by the face inclusions $[n-1] \hookrightarrow [n]$. Let $\iota: \Delta^{inj} \to \Delta$ denote the inclusion functor. We immediately obtain a Quillen adjunction
\[
\iota_!: \Fun(\Delta^{inj},\ssets) \leftrightarrows \Fun(\Delta,\ssets): \iota^*
\]
where $\iota_!$ is the left Kan extension along $\iota$. Note that it is a Quillen adjunction for the projective model structures because the forgetful functor $\iota$ clearly preserves objectwise fibrations and weak equivalences.

By our work above, we already have a Quillen adjunction
\[
{\Delta^{inj}}_{/-}: \ssets \leftrightarrows \Fun(\Delta^{inj},\ssets): \Maps({\Delta^{inj}}_{/-},-). 
\]
Let $\ftot$ be the composition $\Maps({\Delta^{inj}}_{/-},-) \circ \iota^*$. Composing these Quillen adjunctions, we obtain a Quillen adjunction
\[
\iota_! \circ {\Delta^{inj}}_{/-}: \ssets \leftrightarrows \Fun(\Delta,\ssets): \ftot. 
\]
One can check that $\iota_! \circ {\Delta^{inj}}_{/-}$ is weakly equivalent to $C$, and hence $\ftot$ is weakly equivalent to $\holim_\Delta$. 

This construction is called \emph{fat} totalization to contrast it with totalization. The {\it totalization} of $X^\bullet_\bullet \in \cssets$ is the simplicial set given by
\[
\Tot X^\bullet_\bullet := \Maps(\triangle^\bullet_\bullet, X^\bullet_\bullet).
\]
where $\triangle^\bullet_\bullet$ is the {\it cosimplicial standard simplex} whose simplicial set of $n$-cosimplices $\triangle^n_\bullet$ is the standard $n$-simplex $\triangle[n]$. Totalization is thus the dual notion to geometric realization.

\subsection{Yet another definition and its use in the linear setting}

There is another model category structure that is often used on cosimplicial simplicial sets, known as the Reedy model structure (Bousfield and Kan also work with this structure).  For $X^\bullet_\bullet \in \cssets$ that is (co)fibrant in this model structure, we say $X^\bullet_\bullet$ is \emph{Reedy-(co)fibrant}. The cosimplicial standard simplex $\triangle^\bullet_\bullet$ is a cofibrant replacement for the constant functor $C$ in $\cssets$, so we use it to define yet another version the homotopy limit, which we denote by $\holim^R_\Delta$.

\begin{prop}[XI.4.5, \cite{BK}]\label{prop:BK}
For $X^\bullet_\bullet$ a Reedy-fibrant cosimplicial simplicial set, the natural map $\Delta_{/-} \to\triangle^\bullet_\bullet$ induces a weak equivalence
\[
\Tot X^\bullet_\bullet \to \holim^R_\Delta X^\bullet_\bullet.
\]
\end{prop}

Moreover, Bousfield and Kan prove that every cosimplicial simplicial group is Reedy-fibrant (see X.4.9, \cite{BK}). Hence, we can compute homotopy limits using totalization, which is sometimes simpler. 

These results are particularly helpful in connecting the \v{C}ech complex for a sheaf of chain complexes to \v{C}ech descent using $\holim^R$. We use this result to prove our main theorem about $\L8$ spaces, so we explain what we need.

There is a conormalization functor $N^*: cAb \to Ch^+$ that turns a cosimplicial abelian group into a nonnegatively-graded cochain complex of abelian groups. It is the cosimplicial twin of the normalization functor $N_*: sAb \to Ch_+$ appearing in the Dold-Kan correspondence between simplicial abelian groups and nonnegatively-graded chain complexes. Let $\DK: Ch_+ \to sAb$ denote the adjoint functor appearing in the Dold-Kan correspondence.

Given a cosimplicial simplicial abelian group $A^\bullet_\bullet$, let $TA$ denote the total chain complex of the double complex $N^* N_* A^\bullet_\bullet$ (we use the product of groups in making the total complex). Let $\Tot A$ denote the simplicial abelian group given by totalization as a cosimplicial simplicial set. Note that this computes $\holim^R_\Delta A$, as $A$ is fibrant.

\begin{prop}[Lemma 2.2 \cite{BSS}]\label{prop:BSS}
There is a natural quasi-isomorphism $\phi: N_* \Tot A \to TA$. Equivalently, under the Dold-Kan correspondence, there is a natural weak equivalence $\phi': \Tot A \to \DK(TA)$.
\end{prop}

Given a simplicial presheaf $\cF$ such that $\cF(U)$ is a simplicial abelian group for every open $U$, we thus see that $\holim^R_{\check{C}\fV}\,\cF$ is equivalent to the usual \v{C}ech complex for $\cF$. 

\section{Brief overview of the Maurer-Cartan functor}\label{app:MC}

We collect here some definitions and theorems about the Maurer-Cartan space $\MC_\bullet$ that we use in understanding the derived stack $\bB\fg$ from an $\L8$ space $B\fg := (X,\fg)$. The elegant paper  of Getzler \cite{Getzler} is the primary reference. Look there for the proofs and many further insights.

\begin{definition}
The \emph{lower central series} of a $\L8$ algebra $\fg$ over a commutative dg algebra $A$ is the decreasing filtration $F^k \fg$ with $F^1 \fg = \fg$ and 
\[
F^k \fg = \sum_{i_1 + \cdots + i_n = k} \ell_k(F^{i_1} \fg, \ldots, F^{i_n} \fg)
\]
for $k > 1$. (This expression means ``take the span of elements produced by applying the bracket $\ell_k$ to elements from the appropriate level of the filtrations.")

A $\L8$ algebra $\fg$ is \emph{nilpotent} if there is some positive integer $N$ such that $F^N \fg = 0$. In other words, any sufficiently long sequence of brackets vanishes.
\end{definition}

For $\fg$ a nilpotent $\L8$ algebra, the Maurer-Cartan equation $mc(\alpha) = 0$ is well-posed, since the infinite sum --- which is {\it a priori} ill-defined --- is actually a finite sum.

Let $\fg$ denote a nilpotent $\L8$ algebra. Getzler proves several properties that are useful for us.
\begin{enumerate}
\item[1)] $\MC_\bullet(\fg)$ is a Kan complex.
\item[2)] For $f: \fg \to \fh$ a levelwise surjective map of nilpotent $\L8$ algebras, the induced map $\MC_\bullet(f): \MC_\bullet(\fg) \to \MC_\bullet(\fh)$ is a fibration of simplicial sets.
\item[3)] For $\fg$ an abelian $\L8$ algebra (i.e., merely a cochain complex), $\MC_\bullet(\fg)$ is homotopy equivalent to $\DK_\bullet(\trun(\fg[1]))$, where $\trun(\fg[1])$ denotes the brutal truncation where we drop all positive-degree elements of a cochain complex.
\end{enumerate}
These results extend to the nilpotent curved $\L8$ algebras that we work with. 

It might help to know what the Maurer-Cartan space $\MC_\bullet(\fg)$ means when $\fg$ is a nilpotent Lie algebra, in the classical sense, in order to interpret the general construction. The introduction to \cite{Getzler} gives a beautiful explanation, with connections to many topics. He shows, for instance, that for $G$ a simply-connected, nilpotent Lie group and $\fg = Lie(G)$ its associated nilpotent Lie algebra, there is a natural homotopy equivalence $N_\bullet G \to \MC_\bullet \fg$. (Getzler, in fact, goes further and finds a replacement for the Maurer-Cartan functor that ``integrates'' a nilpotent $\L8$ algebra to its ``group.'' For a nilpotent Lie algebra, his functor recovers the nerve of the group on the nose.)

\section{Proof of Theorem \ref{L8IsDerived}}\label{app:proof}

We follow the structure of Costello's argument in \cite{WG2}, reworking and elaborating on several steps of the argument. Let $B\fg = (X,\fg)$ be an $\L8$ space. Recall that $\bB \fg$ denotes the associated functor of points.

The essential idea is to exploit the nilpotent ideal $\sI_\sN = \ker q$ inside $\sO_\cN$ for any nilpotent dg manifold $\cN$ on which we evaluate $\bB\fg$. Let $n$ be the integer such that $\sI_\sN^{n+1} = 0$. Then we have a tower of commutative dg $\Omega^*_N$-algebras
\[
\sO  \to \sO/\sI^n \to \sO/\sI^{n-1} \to \cdots \to \sO/\sI^2 \to \sO/\sI \cong \cinf_N
\]
and an associated tower of nilpotent dg manifolds
\[
N = \cN_1 \to \cN_2 \to \cdots \to \cN_n \to \cN,
\]
where $\cN_k = (N, \sO_\cN/\sI_\cN^k)$. We also obtain a tower of simplicial sets
\[
\bB\fg(\cN) \to \bB\fg(\cN_n) \to \cdots \to \bB\fg(\cN_2) \to \bB\fg(\cN_1) = \bB\fg(N)
\]
for any nilpotent dg manifold $\cN$. 

Our arguments will proceed by induction up this tower (this is a direct analog of artinian induction). For instance, we will show that weak equivalences go to weak equivalences stage by stage along the tower.  

\begin{remark}\label{FixTheMap}
There is an important feature of the functor $\bB\fg$ that we wish to emphasize: the simplicial set $\bB\fg(\cN)$ is a disjoint union of simplicial sets over the set of the smooth maps $f: N \to X$. In the construction of the $k$-simplices $\bB\fg(\cN)_k$, we solve for solutions of the Maurer-Cartan equation in an algebra depending on $\Omega^*(\triangle^k)$ but where the underlying smooth map is independent of $\triangle^k$. {\em From hereon in the proof of the theorem, we fix a map $f$ and simply study solutions over that $f$.} We remind the reader of this assumption periodically, for clarity's sake.
\hfill $\Diamond$ \end{remark}

We start with the base case for the induction.

\begin{lemma}\label{lem:basecaseredux}
For any smooth manifold $N$, the simplicial set $\bB\fg(N)$ is the discrete simplicial set of smooth maps $C^\infty(N,X)$. It is, in particular, a Kan complex.
\end{lemma}

\begin{proof}
See lemma \ref{lem:basecase} above.
\end{proof}

We now show that each stage of the induction is well-behaved homotopically.

\begin{lemma}\label{lem:fibration}
The map 
\[
q: \bB\fg(\cN_{k+1}) \to \bB\fg(\cN_k)
\]
is a fibration.
\end{lemma}

\begin{proof}[Proof of lemma]
Recall that $q$ is a Kan fibration if for any map of an $m$-simplex $s: \triangle^m \to \bB\fg(\cN_{k})$ and any map of a $m$-horn $t: \Lambda^m_j \to \bB\fg(\cN_{k+1})$ such that 
\[
\xymatrix{
\Lambda^m_j \ar[r]^t \ar[d]_\iota & \bB\fg(\cN_{k+1}) \ar[d]^q \\
\triangle^m \ar[r]_s & \bB\fg(\cN_k)
}
\]
is a commutative diagram, we can lift to a map $\tilde{s}: \triangle^m \to \bB\fg(\cN_{k+1})$ such that $t = \tilde{s} \circ \iota$ and $s = q \circ \tilde{s}$.

As explained in remark \ref{FixTheMap}, we are free to fix the underlying smooth map $f: N \to X$ in all these constructions. Let $f: N \to X$ denote that fixed map from hereon. 

Now suppose we have an $m$-simplex $(f,\alpha)$ of $\bB\fg(\cN_{k})$ and an $m$-horn of $\bB\fg(\cN_{k+1})$ agreeing with $(f,\alpha)$ modulo $\sI^k$. We want to extend to a full $m$-simplex.

As a first step, we consider the problem of lifting $\alpha$ simply as an element of {\em graded vector spaces}, ignoring the Maurer-Cartan equation. We have a short exact sequence
\[
0 \to (\sI^k/\sI^{k+1})^\sharp(N) \ot \Omega^\sharp(\triangle^m) \to (\sO/\sI^{k+1})^\sharp(N) \ot \Omega^\sharp(\triangle^m) \overset{Q}{\to} (\sO/\sI^k)^\sharp(N) \ot \Omega^\sharp(\triangle^m) \to 0
\]
of graded vector spaces. We denote this sequence
\[
0 \to K \to B \overset{Q}{\to} A \to 0
\]
for simplicity. Note that $K$ is a square-zero ideal of $B$ and hence a very simple {\em non}unital commutative dg algebra. We eventually want to study solutions to the Maurer-Cartan equation when we tensor $f^*\fg$ with each term in this sequence. Ignoring the differentials for now, observe that we get a sequence of graded vector spaces
\[
0 \to f^* \fg \ot K \to f^*\fg \ot B \overset{{\rm Id} \ot Q}{\longrightarrow} f^*\fg \ot A \to 0.
\]
Thus, given a solution $\alpha$ to the Maurer-Cartan equation in $f^*\fg \ot A$, there exist lifts $\tilde{\alpha}$ in $f^*\fg \ot B$ such that $Q \circ \tilde{\alpha} = \alpha$ because we can split $Q$ as a map of vector spaces. These lifts form a torsor for $f^*\fg \ot K$.

We now ask when such a lift $\tilde{\alpha}$ satisfies the Maurer-Cartan equation. We know $\alpha$ does, so the failure to satisfy the Maurer-Cartan equation lives in $f^*\fg \ot K$. In other words, we have an obstruction living in the second cohomology group of $f^*\fg \ot K$. By hypothesis, we know that $\tilde{\alpha}$ satisfies the Maurer-Cartan equation when restricted to the horn. As $\Omega^*(\Lambda^m_j) \to \Omega^*(\triangle^m)$ is a quasi-isomorphism, we see that the obstruction must vanish.
\end{proof}

This tower of fibrations gives us a procedure for checking a property by working our way up the tower and simply working with $\sI^k/\sI^{k+1}$ at each stage. As this is simply a cochain complex, rather than a nontrivial commutative dg algebra, the problem has become more tractable. 

\begin{prop}\label{prop:weakequiv}
Let $F: \cN \to \cN'$ be a weak equivalence of nilpotent dg manifolds. Then the map $F^*: \bB\fg(\cN') \to \bB\fg(\cN)$ is a weak equivalence of simplicial sets.
\end{prop}

\begin{proof}[Proof of proposition]
We have a diffeomorphism $f: N \to N'$ and a map of commutative dg algebras $\phi: f^{-1} \sO_{\cN'} \to \sO_\cN$ that is compatible with the filtration. Our plan is to show we have a weak equivalence
\[
F^*_k: \bB\fg(N_k') \to \bB\fg(N_k)
\]
for each level $k$ of the tower of nilpotent dg manifolds. 

The base case is straightforward: 
\[
F^*_1 = f^*: \bB\fg(N') \to \bB\fg(N)
\] 
is an isomorphism because $f$ is a diffeomorphism. 

Now suppose $F^*_{k-1}$ is a weak equivalence. To show $F^*_k$ is a weak equivalence, it suffices to show that the induced map on fibers is a weak equivalence by the long exact sequence in homotopy groups. Here we use remark \ref{FixTheMap}: it suffices to study the problem for each distinct smooth map $g: N' \to C$. This choice then fixes the smooth map $g \circ f: N \to X$. We will compare the fibers over $g$ and $h = g \circ f$ but suppress them from discussion.

The fiber for $N$ is then the Maurer-Cartan space for $h^* \fg \ot_{\Omega^*_N} \sI^{k-1}_\cN/\sI^k_\cN$. This curved $\L8$ algebra is, in fact, abelian because the $\sI^{k-1}/\sI^k$ is square-zero, so this simplicial set has another name: it arises from the cochain complex $h^* \fg \ot_{\Omega^*_N} \sI^{k-1}_\cN/\sI^k_\cN[1]$ under the Dold-Kan correspondence. 

By the hypothesis, we know that $\Gr \phi$ is a quasi-isomorphism. This means that the cochain complex $\sI^{k-1}_\cN/\sI^k_\cN$ is quasi-isomorphic to $\sI^{k-1}_{\cN'}/\sI^k_{\cN'}$ for every $k$. Hence the fibers --- as simplicial sets constructed by the Dold-Kan correspondence --- are weakly equivalent.
\end{proof}

It remains to show that $\bB\fg$ satisfies \v{C}ech descent. This argument is the trickiest because it involves homotopy limits.

\begin{prop}\label{prop:descent}
The simplicial presheaf $\bB\fg$ satisfies \v{C}ech descent.  That is, for $\fV$ any cover of $\cN$, the map
\[
\bB\fg (\cN) \to \holim_{\check{C}\fV} \bB\fg = \holim \check{C}^\bullet (\fV , \bB\fg)
\]
is a weak equivalence.
\end{prop}

This argument also proceeds by artinian induction. We fix useful notation. Let $\fV = \{ V_i\}$ denote a cover of a smooth manifold $N$, and let $\fV_k = \{ (V_i, \sO_{\cN_k} |_{V_i})\}$ denote the associated cover of the nilpotent dg manifold $\cN_k$.

The base case of the argument is the following standard fact.

\begin{lemma}\label{lem:baseofcech}
The natural map $\bB\fg(N) \to \holim_{\check{C}\fV} \bB\fg$ is a weak equivalence of simplicial sets.
\end{lemma}

\begin{proof}
By lemma \ref{lem:basecaseredux}, we know that all the simplicial sets here are discrete. In particular, there is the associated map of sets 
\[
C^\infty(N,X) \to \holim {\check{C}(\fV, C^\infty(-,X))}.
\]
Hence the homotopy limit agrees with the limit, and we know that $C^\infty(-,X)$ is a sheaf of sets on $\Man$, so that map of sets is an isomorphism. 
\end{proof}

With the base case behind us, we tackle the induction step. Observe that we have a commuting square of simplicial sets
\[
\xymatrix{
\bB\fg(\cN_k) \ar[d] \ar[r] &  \holim_{\check{C}\fV_k} \bB\fg \ar[d]\\
\bB\fg(\cN_{k-1}) \ar[r] &  \holim_{\check{C}\fV_{k-1}} \bB\fg
}
\]
where the left vertical map is a fibration and the bottom horizontal map is a weak equivalence. We will show that the right vertical map is a fibration and that the induced map on fibers is a weak equivalence. This then implies that the top horizontal map is a weak equivalence, finishing the induction step.

\begin{lemma}
The map
\[
\check{C}(\fV_k, \bB\fg) \to \check{C}(\fV_{k-1}, \bB\fg)
\]
is a fibration of cosimplicial simplicial sets. As $\holim_\Delta: cs\!\Sets \to s\!\Sets$ preserves fibrations,\footnote{It is the {\em right} adjoint in a Quillen adjunction
\[
c: s\!\Sets \leftrightarrows cs\!\Sets: \holim
\]
as discussed in proposition \ref{prop:universalprop}.} we see that
\[
\holim_{\check{C}\fV_k} \bB\fg \to \holim_{\check{C}\fV_{k-1}} \bB\fg
\]
is a fibration of simplicial sets.
\end{lemma}

\begin{proof}
We already showed in lemma \ref{lem:fibration} that $\bB\fg$ is a fibration along any square-zero extension. Hence, we know the map of \v{C}ech diagrams is an objectwise fibration and hence a fibration in the projective model structure on $\cssets$. Thus, the homotopy limit, as a right Quillen adjoint, gives us a fibration of simplicial sets.
\end{proof}

It remains to show that the map of fibers is a weak equivalence. We know that the fiber of the map $\bB\fg(\cN_k) \to \bB\fg(\cN_{k-1})$ is simply
\[
\MC_\bullet(\fg \ot \sI^{k-1}/\sI^{k}(N)),
\]
the Maurer-Cartan space for the $\L8$ algebra $\fg \ot \sI^{k-1}/\sI^{k}(N)$. (We include the notation ``$(N)$" to emphasize that we are taking global sections of this sheaf of $\L8$ algebras.) On the other side of the square above, we know that $\holim$ preserves fibrations, so that the fiber of the map $\holim_{\check{C}\fV_k} \bB\fg \to \holim_{\check{C}\fV_{k-1}} \bB\fg$ is given by
\[
\holim \check{C}^\bullet (\fV, \MC_\bullet(\fg \ot \sI^{k-1}/\sI^{k})),
\]
namely the \v{C}ech cosimplicial diagram for the cover $\fV$ with respect to the simplicial presheaf that assigns the Maurer-Cartan space for the sheaf of $\L8$ algebras $\fg \ot \sI^{k-1}/\sI^{k}$.

\begin{lemma}
The map between the fibers
\[
\MC(\fg \ot \sI^{k-1}/\sI^{k})(N) \to \holim \check{C}^\bullet (\fV, \MC(\fg \ot \sI^{k-1}/\sI^{k})) 
\]
is a weak equivalence.
\end{lemma}

This lemma concludes the proof of proposition \ref{prop:descent}.

\begin{proof}
Recall that for an abelian $\L8$ algebra $\mathfrak{h}$, the Maurer-Cartan space $\MC_\bullet(\frak{h})$ is just the Dold-Kan space $\DK_\bullet(\trun(\frak{h}[1]))$. Note that the Dold-Kan space only depends on the truncation $\trun(\frak{h}[1])$, where all components of $\fh$ of degree greater than 1 are eliminated. In particular, we see that
\[
\MC_\bullet(\fg \ot \sI^{k-1}/\sI^{k}(N)) = \DK_\bullet(\trun(\fg \ot \sI^{k-1}/\sI^{k}(N)[1])),
\]
so that the fiber is determined by a cochain complex.

As $\fg \ot \sI^{k-1}/\sI^{k}$ is a sheaf of abelian $\L8$ algebras, we are simply applying $\DK_\bullet$ to the value on every open. Hence $\check{C}^\bullet (\fV, \DK_\bullet(\trun(\fg \ot \sI^{k-1}/\sI^{k}[1])))$ is a cosimplicial simplicial abelian group. By proposition \ref{prop:BSS}, we know that $\holim \check{C}^\bullet (\fV, \MC_\bullet(\fg \ot \sI^{k-1}/\sI^{k}))$ is thus weakly equivalent to the Dold-Kan simplicial set of  
\[
T\check{C}(\fV,\trun(\fg \ot \sI^{k-1}/\sI^{k}[1])),
\]
the total complex formed from the \v{C}ech double complex for the sheaf $\trun(\fg \ot \sI^{k-1}/\sI^{k}[1])$ on the cover $\fV$. 

Thus, to prove the lemma, we verify instead that
\[
\trun(\fg \ot \sI^{k-1}/\sI^{k})(N)[1]) \to T\check{C}(\fV,\trun(\fg \ot \sI^{k-1}/\sI^{k}[1]))
\]
is a quasi-isomorphism, as this implies that the Dold-Kan simplicial sets are weakly equivalent.

This map is a quasi-isomorphism because $\fg \ot \sI^{k-1}/\sI^{k}$ consists of smooth sections of a vector bundle on $N$, so the usual partition of unity arguments apply.
\end{proof}

\bibliographystyle{amsalpha}
\bibliography{current}

\end{document}